\documentclass[11pt]{amsart}
	
\usepackage{amsfonts,amsmath,amsthm,amssymb,}
\usepackage[usenames,dvipsnames,svgnames,table]{xcolor}
\usepackage{todonotes}
\usepackage{mathrsfs}

\usepackage[margin=1.4in]{geometry} %I am using this instead of fullpage so that we can see our todo notes

\usepackage{todonotes}
\usepackage{enumitem} %This is supposed to give some better control over the old enumerate package.

\numberwithin{equation}{section}

\newtheorem{thm}{Theorem}[section]
\newtheorem*{thm*}{Theorem}
\newtheorem{prop}[thm]{Proposition}
\newtheorem*{prop*}{Proposition}
\newtheorem{cor}[thm]{Corollary}
\newtheorem{lemma}[thm]{Lemma}

\newtheorem*{question*}{Question}

\theoremstyle{definition}
\newtheorem{defin}[thm]{Definition}

\newtheorem{nota}[thm]{Notation}
\newtheorem{notarem}[thm]{Notation and Remarks}

\newtheorem{example}[thm]{Example}
\newtheorem{remark}[thm]{Remark}
\newtheorem*{remark*}{Remark}
\newtheorem*{remarks*}{Remarks}

\newtheorem{main}[thm]{Main Theorem}
\newtheorem{open}[thm]{Open Question}

\newcommand{\Aut}{\operatorname{Aut}}
\newcommand{\mytrace}{\operatorname{trace}}
\newcommand{\Ric}{\operatorname{Ric}}
\newcommand{\nilrad}{\operatorname{Nilrad}}
\newcommand{\rad}{\operatorname{Rad}}
\newcommand{\Der}{\operatorname{Der}}
\newcommand{\ad}{\operatorname{ad}}
\newcommand{\Ad}{\operatorname{Ad}}

\newcommand{\Id}{\operatorname{Id}}

\newcommand{\Inn}{\operatorname{Inn}}
\newcommand{\SL}{\operatorname{SL}}

\newcommand{\s}{\mathfrak{s}}

\newcommand{\g}{\mathfrak{g}}
\newcommand{\n}{\mathfrak{n}}
\newcommand{\rf}{\mathfrak{r}}
\newcommand{\lf}{\mathfrak{l}}
\newcommand{\ff}{\mathfrak{f}}
\newcommand{\mf}{\mathfrak{m}}
\newcommand{\h}{\mathfrak{h}}

\newcommand{\kf}{\mathfrak{k}}
\newcommand{\p}{\mathfrak{p}}
\newcommand{\m}{\mathfrak{m}}
\newcommand{\uf}{\mathfrak{u}}
\newcommand{\vf}{\mathfrak{v}}
\newcommand{\w}{\mathfrak{w}}
\newcommand{\af}{\mathfrak{a}}
\newcommand{\cf}{\mathfrak{c}}
\newcommand{\slf}{\mathfrak{sl}}

\newcommand{\tK}{\widetilde{K}}
\newcommand{\tG}{\widetilde{G}}
\newcommand{\tf}{\mathfrak{t}}
\newcommand{\bfrak}{\mathfrak{b}}

\newcommand{\R}{\mathbf{R}}
\newcommand{\Z}{\mathbf{Z}}
\newcommand{\C}{\mathbf{C}}

\newcommand{\D}{\displaystyle}

\newcommand{\Rcal}{{\mathcal R}}
\newcommand{\Scal}{{\mathcal S}}

\newcommand{\Lc}{{\mathcal L}}

\newcommand{\Isom}{{\operatorname{Isom}}}
\newcommand{\cisom}{\mathfrak{isom}}

\newcommand{\tL}{{\widetilde{L}}}

\newcommand{\adg}{\ad_\g}
\newcommand{\orth}{{\operatorname{orth}}}
\newcommand{\sk}{{\operatorname{skew}}}

\newcommand{\M}{\mathcal{M}}

\newcommand{\Adop}{{\operatorname{Ad}}}
\newcommand{\B}{\mathcal{B}}
\newcommand{\SO}{\rm{SO}}
\newcommand{\Diff}{{\operatorname{Diff}}}
\newcommand{\spa}{{\operatorname{span}}}
\begin{document}

\title{Infinitesimal maximal symmetry and Ricci soliton solvmanifolds}

\author[C.S. Gordon]{Carolyn S. Gordon}
\address{Department of Mathematics, Dartmouth College, Hanover, NH 03755-1808, USA}
\email{carolyn.s.gordon@dartmouth.edu}

\author[M.R. Jablonski]{Michael R. Jablonski$^2$}
\address{Department of Mathematics, University of Oklahoma, Norman, OK 73019-3103, USA}
\email{mjablonski@math.ou.edu}
\thanks{$^2$Research partially supported by National Science Foundation
grant DMS-1906351}
%\thanks{%MSC2010: 53C25, 53C30, 22E25}

\maketitle

\begin{abstract}

This work addresses the questions:  (i)  Among all left-invariant Riemannian metrics on a given Lie group, is there any whose isometry group  or isometry algebra  contain that of all others?  (ii) Do expanding left-invariant Ricci solitons exhibit such maximal symmetry?  Question (i) is addressed both for semisimple and for solvable Lie groups.     Building on previous work of the authors on Einstein metrics, a complete answer is given to (ii):  expanding homogeneous Ricci solitons have maximal isometry algebras although not always maximal isometry groups.  

As a consequence of the tools developed to address these questions,  partial results of B\"ohm, Lafuente, and Lauret are extended to show that left-invariant Ricci solitons on solvable Lie groups are unique up to scaling and isometry.

\end{abstract}

%The below makes the To Do list in our document
%\makeatletter
%\providecommand\@dotsep{5}
%\makeatother

\tableofcontents

\section{Introduction}

What are the ``nicest'' left-invariant  Riemannian metrics on a given connected Lie group?   Some metrics may stand out due to their special curvature properties; others may be notable for their symmetry properties.     One may ask:  

\begin{itemize}
\item[(Q1)] On a given Lie group $G$, do there exist metrics that maximize symmetry in some sense?  
\item[(Q2)] Do the metrics with the ``nicest'' curvature properties have maximal symmetry?  
\end{itemize}

The current paper is a sequel to the article \cite{GordonJablonski:EinsteinSolvmanifoldsHaveMaximalSymmetry}, which introduced a notion of maximally symmetric Riemannian metric, obtained both positive and negative results on (Q1), and answered (Q2) affirmatively for Einstein metrics on solvable Lie groups but negatively for more general Ricci solitons on solvable Lie groups.    Motivated by the negative results in  \cite{GordonJablonski:EinsteinSolvmanifoldsHaveMaximalSymmetry} and by interest in the symmetry properties of Ricci solitons, we will introduce weaker notions of maximal symmetry that yield much stronger affirmative answers to (Q1) and (Q2).

\subsection{Previous work.} The article \cite{GordonJablonski:EinsteinSolvmanifoldsHaveMaximalSymmetry} introduced the strongest possible notion of \emph{maximally symmetric} left-invariant Riemannian metric:
 
 \begin{defin}\label{prelimdef} A left-invariant Riemannian metric $g$ on a Lie group $G$ is said to be \emph{maximally symmetric} if its isometry group $\Isom(G,g)<\Diff(G)$ contains that of every other left-invariant metric on $G$ (up to conjugation in $\Diff(G)$ by an element of $\Aut(G$)).
 
\end{defin}
 
Concerning question (Q1), we found, for example, that all semisimple Lie groups of non-compact type, all compact simple Lie groups, and all simply-connected completely solvable unimodular Lie groups admit left-invariant metrics of maximal symmetry.   In contrast, we showed that some Lie groups do not admit maximally symmetric left-invariant metrics.   Counterexamples even include some simply-connected compact semisimple (but not simple) Lie groups, even though such Lie groups admit bi-invariant metrics that give them the structure of a Riemannian symmetric space. 

Concerning question (Q2), we proved:

\begin{thm}\label{oldmain} \cite{GordonJablonski:EinsteinSolvmanifoldsHaveMaximalSymmetry} Every left-invariant Einstein metric on a simply-connected solvable Lie group is maximally symmetric.
\end{thm}

The proof of Theorem~\ref{oldmain} involved a blend of Lie group structure theory and geometric invariant theory.    Shortly after,  C. B\"ohm and R. Lafuente \cite{Bohm-Lafuente:TheRicciFlowOnSolvmanifoldsOfRealType} showed that left-invariant Einstein metrics $g_0$ on solvable Lie groups are stable fixed points (up to automorphism) of the scalar-curvature normalized Ricci flow; i.e., the flow starting from any left-invariant metric smoothly converges to $g_0$ up to pullback by an automorphism.  As a corollary, they obtained a new proof of Theorem~\ref{oldmain}.  

In 1975, Alekseevskii conjectured that every homogeneous Einstein manifold of non-positive Ricci curvature is diffeomorphic to Euclidean space; an a priori slightly stronger version of the statement was that every such manifold is isometric to a simply-connected solvmanifold, i.e., a simply-connected solvable Lie group with a left-invariant metric.      Motivated in part by this conjecture, the past twenty-five years has seen the development of a beautiful structure theory of Einstein solvmanifolds; see e.g., \cite{Heber}, \cite{LauretStandard,
LauretLafuente:StructureOfHomogeneousRicciSolitonsAndTheAlekseevskiiConjecture,
JP:TowardsTheAlekseevskiiConjecture,
Arroyo-Lafuente:TheAlekseevskiiConjectureInLowDimensions}.  In 2018, C. B\"ohm and R. Lafuente \cite{Bohm-Lafuente:HomogeneousEinsteinMetricsOnEuclideanSpacesAreEinsteinSolvmanifolds} proved that the Alekseevskii conjecture would indeed imply that every homogeneous Einstein manifold of non-positive Ricci curvature is a solvmanifold and, in a major breakthrough  \cite{BL2021} a few years later, they proved the full Alekseevskii conjecture.  See \cite{Jablo:SurveyHomogeneousEinsteinManifolds}, by the second author, for a recent survey on homogeneous Einstein spaces.

Einstein metrics are stable under the Ricci flow.  More generally, Ricci solitons are metrics that evolve only by pullback by diffeomorphisms and by homothety under the Ricci flow.    The study of expanding homogeneous Ricci solitons has seen great advances mirroring the advances on Einstein solvmanifolds, see  \cite{Jablo:HomogeneousRicciSolitonsAreAlgebraic}.   Putting together the two works \cite{Jablo:HomogeneousRicciSolitonsAreAlgebraic} and  \cite{LauretLafuente:StructureOfHomogeneousRicciSolitonsAndTheAlekseevskiiConjecture}, one sees that the Alekseevskii conjecture implies that every expanding homogeneous Ricci soliton is also a solvmanifold.

A counterexample in \cite{GordonJablonski:EinsteinSolvmanifoldsHaveMaximalSymmetry} shows that Theorem~\ref{oldmain} cannot be extended to the class of left-invariant Ricci solitons on simply-connected solvable Lie groups, or even to so-called ``solvsolitons''.   \emph{Solvsolitons} are a special class of left-invariant Ricci solitons on solvable Lie groups, defined by the condition that they evolve only by pullback by \emph{Lie group automorphisms} (as opposed to more general diffeomorphisms) and homotheties under the Ricci flow.  However, the Ricci soliton in the counterexample did exhibit a weaker symmetry property that motivated the results of the current paper.

\subsection{New results.}  

We introduce two successively weaker symmetry conditions, almost maximal symmetry and infinitesimal maximal symmetry,    that exist on successively larger classes of Lie groups.   The first is identical to the notion of maximal symmetry except that only the identity component of the isometry group is considered.    The second considers only the isometry algebras.

 \begin{defin} Let $g$ be a left-invariant metric on a connected Lie group $G$.
 \begin{enumerate}
 \item $g$ is said to be \emph{almost maximally symmetric} if the identity component $\Isom_0(G,g)$ of its isometry group contains that of every other left-invariant metric on $G$ (up to conjugation in $\Diff(G)$ by an element of $\Aut(G$)).
 
 \item $g$ is said to be \emph{infinitesimally maximally symmetric} if the Lie algebra of its isometry group contains an isomorphic copy of that of every other left-invariant metric.
 \end{enumerate}
 
\end{defin}

\subsubsection{Addressing Question (Q1).}

 We first show that each successive notion -- maximal, almost maximal, and infinitesimal maximal symmetry -- is indeed weaker than the previous one.    For example, all left-invariant metrics on 2-dimensional tori are almost maximally symmetric but none are maximally symmetric.    Bi-invariant Riemannian metrics on compact semisimple Lie groups are always infinitesimally maximally symmetric but not always almost maximally symmetric.   On the other hand, we show that when a connected Lie group $G$ admits at least one almost maximally symmetric metric, then \emph{every} infinitesimally maximally symmetric metric on $G$ is necessarily almost maximally symmetric as well.  
           
     We focus especially on the setting of simply-connected solvable Lie groups.   Motivated by results of \cite{GordonWilson:IsomGrpsOfRiemSolv,Jablo:MaximalSymmetryAndUnimodularSolvmanifolds}, we introduce in Theorem ~\ref{thm: equiv classes} an equivalence relation on the class of all simply-connected solvable Lie groups.  Even though the equivalence relation is defined purely in terms of the structure of the associated Lie algebras independently of any metric, it has significant geometric implications.  The following theorem is a compendium of Theorems ~\ref{thm: equiv classes}, \ref{thm.R admits max sym} and \ref{thm: solitons and equiv classes}:  (Recall that a Lie group $S$ is \emph{completely solvable} if for every $X$ in its Lie algebra, all eigenvalues of $\ad(X)$ are real.)

       \begin{thm}\label{compendium} There exists an equivalence relation on the collection of all simply-connected solvable Lie groups, defined explicitly in Theorem~\ref{thm: equiv classes} by a purely Lie algebraic condition, which satisfies the following: 
              \begin{enumerate}
     \item Each equivalence class contains a unique (up to isomorphism) completely solvable Lie group.
     \item Two simply-connected solvable Lie groups $R_1$ and $R_2$ are equivalent if and only if there exists left-invariant metrics $g_1$ on $R_1$ and $g_2$ on $R_2$ such that $(R_1,g_1)$ is isometric to $(R_2,g_2)$.  
     \item If a simply-connected solvable Lie group $R$ admits a left-invariant Ricci soliton $g$, then every Lie group in the equivalence class of $R$ admits a left-invariant Ricci soliton isometric to $g$.
     \item If a completely solvable simply-connected Lie group $S$ admits an infinitesimally maximally symmetric metric $g_0$, then every Lie group in the equivalence class of $S$ admits an infinitesimally maximally symmetric metric isometric to $g_0$.  
     \end{enumerate}   
         \end{thm}
 The first three statements are interpretations or elementary consequences of the literature on the structure of Riemannian solvmanifolds and on Ricci solitons.   
 
Example~\ref{R not admit almost max} shows that the final statement fails if we replace infinitesimal maximal symmetry by almost maximal symmetry; i.e., if the completely solvable Lie group $S$ admits an almost maximally symmetric metric $g_0$, the Lie groups in the equivalence class will admit infinitesimally maximally symmetric, but in general not maximally symmetric, metrics isometric to $g_0$.

 In the special case of completely solvable Lie groups $S$ with trivial center, we show in Proposition~\ref{comp solv inf implies almost} that \emph{every} infinitesimally maximally symmetric metric is necessarily almost maximally symmetric.

While the results of this paper show that not all Lie groups admit metrics of almost maximal symmetry, the following question appears challenging:

\begin{open} Does every Lie group admit an infinitesimally maximally symmetric metric?
\end{open}
    We are continuing to investigate this question.  For some progress, see  \cite{EpsteinJablonski:MaximalSymmetryAndSolvmanifolds} by J. Epstein and the second author.
    
\subsubsection{Application to uniqueness questions for Ricci solitons}\label{applics}

Several authors have addressed questions of uniqueness of left-invariant Ricci solitons on simply-connected solvable Lie groups $S$:

\begin{itemize} \item[$\bullet$] [Lauret  \cite{Lauret:SolSolitons}] A simply-connected solvable Lie group can admit at most one solvsoliton up to homothety and isometry.
\item[$\bullet$][Jablonski \cite{Jablo:HomogeneousRicciSolitons}] Every left-invariant Ricci soliton on a simply-connected \emph{completely solvable} Lie group is necessarily a solvsoliton.  (Thus there can exist at most one soliton up to homothety and isometry.)
\item[$\bullet$] [B\"ohm and R. Lafuente \cite{Bohm-Lafuente:TheRicciFlowOnSolvmanifoldsOfRealType}] If an arbitrary simply-connected solvable Lie group $S$ admits a non-flat solvsoliton $g$, then $g$ is the unique left-invariant Ricci soliton on $S$ up to homothety and pullback by automorphisms. 
\end{itemize} 

While the final result above strengthened the first, it left open the question of uniqueness of left-invariant Ricci solitons on arbitrary simply-connected solvable Lie groups  (in particular, those solvable Lie groups that admit left-invariant Ricci solitons but not solvsolitons).  As an immediate consequence of Theorem~\ref{compendium} and the uniqueness results above for completely solvable Lie groups, we have:

\begin{cor}\label{cor.uniqueness}  A simply-connected solvable Lie group can admit at most one left-invariant Ricci soliton up to homothety and isometry.
\end{cor}

We emphasize that the uniqueness property ``up to isometry'' in the corollary is weaker than ``up to pullback by automorphisms" in general.     The conclusion of the corollary is sharp however; in Example~\ref{R not admit almost max} we construct a simply-connected solvable Lie group $S$ that admits two isometric left-invariant Ricci solitons that are not related by an automorphism of $S$.     

\subsubsection{Addressing Question (Q2) for Ricci solitons.}    

We prove:  

\begin{main}\label{main}~
\begin{enumerate}
\item Every left-invariant Ricci soliton metric $g$ on a simply-connected solvable Lie group $S$ is infinitesimally maximally symmetric.
\item If, moreover, $S$ is completely solvable with trivial center, then $g$ is almost maximally symmetric.
\end{enumerate}
\end{main}

We emphasize that the theorem applies to \emph{all} left-invariant Ricci solitons, not just to solvsolitons.        
 
     Proposition~\ref{comp solv inf implies almost}, quoted above, shows that the second statement of the main theorem is a consequence of the first one.  
 
 The key tools we use to prove the first statement are the following:  By Theorem~\ref{compendium}, it suffices to consider the case that $S$ is completely solvable, in which case (by the results cited in~\ref{applics} above) the Ricci soliton $g$ is necessarily a solvsoliton and is unique up to automorphism and homothety.    Moreover, Lauret \cite{Lauret:SolSolitons} showed that $(S,g)$ admits a one-dimensional extension $(S_E,g_E)$ to an Einstein solvmanifold.  The Lie algebra $\s_E$ is the one-dimensional extension of $\s$ arising from the so-called \emph{pre-Einstein derivation} of $\s$ (more precisely, of the nilradical of $\s$) introduced by Y. Nikolayevsky \cite{Nikolayevsky:EinsteinSolvmanifoldsandPreEinsteinDerivation}.

We use these tools to reduce Theorem~\ref{main} to Theorem~\ref{oldmain} as follows:  Letting $\g$ be the isometry algebra of an arbitrary left-invariant metric on $S$, the key step in the proof is to show that the pre-Einstein derivation of the nilradical of $\s$ extends to a derivation of $\g$.    This allows us to construct a one-dimensional extension $\g'$ of $\g$ and a corresponding linear Lie group $G'$ that acts transitively on the extension $S_E$ of $S$.   We apply Theorem~\ref{oldmain} to conclude that $G'$ preserves the (unique up to automorphism) Einstein metric $g_E$, and the subgroup with Lie algebra $\g$ preserves (the pullback by an automorphism of) the Ricci soliton $g$.   Thus $\g$ is isomorphic to a subalgebra of the isometry algebra of $(S, g)$ and the theorem follows.

As a consequence of the proof of the main theorem above, we have the following result of independent interest.  

\begin{cor} Let $S$ be a completely solvable Lie group with solvsoliton metric $g$.  Let $(S_E, g_E)$ be the rank-1 Einstein extension, as above.  Then $\Isom_0(S,g)$ is precisely the connected group of isometries of $(S_E, g_E)$ which preserve $S$.
\end{cor}

It would be interesting to know if  the result holds not only for the identity component  but also for the full isometry group.

\subsection*{Acknowledgements} We would like to thank Joseph Wolf for very helpful discussions.

\section{Riemannian solvmanifolds and equivalence classes of solvable Lie groups}

A connected Riemannian manifold $\M$ is said to be a \emph{Riemannian solvmanifold} if it admits a transitive solvable group of isometries.  Every Riemannian solvmanifold  admits an almost simply-transitive solvable group $R$ of isometres.  (Almost simply-transitive means transitive with discrete isotropy.)  If $\M$ is simply-connected, then every such $R$ acts simply transitively.   In this case, the  Riemannian structure induces a left-invariant metric $g$ on $R$ so that $\M\simeq (R,g)$.  

Given a connected Riemannian manifold $\M$, we denote by $\Isom(\M)$ the full isometry group of $\M$.  The identity component of $\Isom(\M)$ will be denoted $\Isom_0(\M)$, and the Lie algebra of $\Isom(\M)$ will be denoted $\cisom(\M)$.

We will always denote Lie groups by capital Roman letters and their Lie algebras by the corresponding fraktur.  For example, the Lie algebra of $R$ is denoted $\mathfrak r$.

\begin{notarem}\label{isomgrps} 

Given a connected Lie group $H$, we denote by $\Lc(H)$ the space of all left-invariant Riemannian metrics on $H$.   Note that $\Lc(H)$ corresponds to the set of all inner products on the Lie algebra $\h$.  

For $g\in \Lc(H)$, denote by $\lambda^g: H\to\Isom(H,g)$ the monomorphism $h\to L_h\in \Isom(H,g)$, where $L_h: H\to H$ is left translation.    We will usually identify the image of $\lambda^g$ with $H$ itself.  The normalizer of $H$ (i.e., of  $\lambda^g(H)$) in $\Isom(H,g)$, respectively of $\h$ (identified with $\lambda^g_*(\h))$ in $\cisom(H,g)$, is given by $H\rtimes \Aut_\orth(H)$, respectively $\h\rtimes\Der_{\sk}(\h,g
)$, where $\Aut_\orth(H)$ is the group of all automorphisms of $H$ that preserve $g$ and $\Der_{\sk}(\h,g)$ is the space of all derivations of $\h$ that are skew-symmetric relative to the inner product defined by $g$.    
\end{notarem}

\subsection{Background on Riemannian solvmanifolds and their isometry groups}

The first systematic study of isometry groups of Riemannian solvmanifolds $\M$ was carried out in \cite{GordonWilson:IsomGrpsOfRiemSolv}.  
There a notion of modification was developed that yields a simple two-step algorithm to obtain a distinguished, almost simply-transitive, solvable Lie subgroup $S$ of $\Isom(\M)$, called a subgroup in ``standard position'' in $\Isom(\M)$.  The subgroup $S$ is distinguished by the property of having maximal normalizer; more precisely, its  normalizer in $\Isom(\M)$ contains --up to conjugacy -- the normalizer of every other almost simply-transitive solvable subgroup of $G$.   In some cases, e.g., when $S$ is unimodular,  one is guaranteed that $S$ is normal in $\Isom(\M)$.  In such cases $\Isom(M)$ can be  explicitly computed.   In more general settings, one still obtains considerable information about $\Isom(\M)$ from knowledge of $S$.   We will briefly review the notions of modification and standard position.

To motivate the theory of modifications and our applications in this work, we present a simple example:

\begin{example}
Consider $M = \mathbb R^3$ with a flat metric.  The simply-connected solvable Lie group $R = \widetilde{SO(2)}\ltimes \mathbb R^2$ acts simply transitively by isometries on $M$ and thus inherits a flat left-invariant metric $g$.   The action of $\widetilde{SO(2)}$  is by rotation on the xy-coordinates as one traverses the z-axis -- a twisting action -- while $\R^2$ acts by translations parallel to the $xy$-plane.    If one begins with $(R,g)$, the standard modification algorithm described below yields the more natural simply-transitive group $\R^3$.
\end{example}

\begin{defin}\label{def.mod}\cite{GordonWilson:IsomGrpsOfRiemSolv} Let $G$ be a Lie group and let $R$ and $R'$ be solvable subgroups of $G$.    
We say that $R'$ is a \emph{modification} of $R$ in $G$ and that the corresponding Lie algebra $\rf'$ is a \emph{modification} of $\rf$ in $\g$ if there exists a linear map $\varphi:\rf\to\g$ such that
\begin{enumerate}
\item $\rf'=\{X+\varphi(X): X\in\rf\}$.
\item  The image of $\varphi$ lies in a compactly embedded subalgebra of $\g$;
\item $[\varphi(\rf),\rf]<\rf$.
\end{enumerate}
 $\varphi$ is said to be a \emph{modification map}.   The fact that $\rf'$ is solvable implies  that $\varphi(\rf)$ is abelian.

The modification is said to be \emph{normal} if $\varphi(\rf)$ normalizes $\rf'$.   Normality says that the modification is reversible, i.e.,  $\rf$ is a modification of $\rf'$, with modification map $\psi(X+\varphi(X))=-\varphi(X)$.   
\end{defin}

We note that the concept of normal modifications appeared independently in the work of Alekseevskii where they are called twists, see  \cite{Alekseevski:RiemSpacesOfNegCurv}.

\begin{prop}\label{modprops}
\cite[Lemma 4.3]{Jablo:MaximalSymmetryAndUnimodularSolvmanifolds}  Every modification of a completely solvable Lie algebra is normal.
\end{prop}

\begin{prop}\label{prop: modification picking up metric} Let $R$ be a simply-connected solvable Lie group, let $U$ be a compact connected subgroup of $\Aut(R)$, and let $H=R\rtimes U$.   Let $R'$ be a connected solvable subgroup of $H$. Then $R'$ is a modification of $R$ in $H$ if and only if $R'$ acts simply transitively on $H/U$.    Moreover, these conditions imply that $R'$ is simply-connected.

\end{prop}

\begin{proof} First assume that $R'$ is a modification of $R$ in $H$.  It is shown by an elementary argument in \cite{GordonWilson:IsomGrpsOfRiemSolv} (without the hypothesis of simply-connectivity of $R$) that $R'$ must act almost simply transitively on $H/U$.   The fact that $H/U$ is simply-connected in our case implies that the action must be simply-transitive and that $R'$ must be simply-connected.

Conversely, if $R'$ acts simply transitively, then $\rf'$ is necessarily a complement to the Lie algebra $\uf $ of $U$ in $\h$.  The assertion that $\rf'$ is a modification of $\rf$ in $\h$ (and thus that $R'$ is a modification of $R$ in $H$) thus follows immediately from the fact that $\uf$ is compactly embedded in $\h$.
\end{proof}

\begin{remark}\label{stdmodalg} Let $\M$ be a simply-connected Riemannian solvmanifold and let $R$ be a simply-transitive group of isometries of $\M$.   By Remark~\ref{isomgrps}, we have $\rf\rtimes\Der_{\sk}(\rf)<\cisom(\M)$.   As shown in \cite{GordonWilson:IsomGrpsOfRiemSolv}, the orthogonal complement $\rf'$ of $\Der_{\sk}(\rf)$ in $\rf\rtimes\Der_{\sk}(\rf)$ relative to the Killing form is a modification of $\rf$, referred to as the \emph{standard modification} (relative  to the Riemannian structure).     The so-called \emph{standard modification algorithm} starts with $\rf$ and takes successive standard modifications.   It is shown in \cite{GordonWilson:IsomGrpsOfRiemSolv} that this algorithm stabilizes in at most two steps.  The resulting Lie algebra $\s$ and the corresponding simply-transitive subgroup $S$ of $\Isom_0(\M)$ are said to be in \emph{standard position}.

 \end{remark}

\begin{prop}\label{stdpos}
\cite{GordonWilson:TheFineStructureOfTransitiveRiemannianIsometryGroups,
GordonWilson:IsomGrpsOfRiemSolv}  Let $\M$ be a simply-connected Riemannian solvmanifold, let $G=\Isom_0(\M)$ and let $G=G_1G_2$ be a Levi decomposition of $G$, where $G_1$ is semisimple and $G_2$ is the solvable radical of $G$.  Let $G_1=KA_1N_1$ be an Iwasawa decomposition of $G_1$, and let $M$ be the normalizer of $A_1$ in $K$.  Set $F=(MA_1N_1)G_2$.   Let $K=K_1K_2$ and $M=M_1M_2$ be the Levi decompositions of the reductive groups $K$ and $M$ respectively.  (In particular, $M_1<K_1$).  Then:
\begin{enumerate}
\item There exists a point $p\in\M$ such that the isotropy subgroup $L$ of $G$ at $p$ satisfies
\begin{itemize}
\item $K_1<L<KG_2$,
\item $L$ normalizes $G_1$, and $G_1L$ is a reductive subgroup of $G$.
\end{itemize}
\item Let $S$ be the simply-connected subgroup of $F$ whose Lie algebra $\s$ is the orthogonal complement of $\lf\cap \ff$ relative to the Killing form of $\ff$.   Then $S$ is a simply-transitive solvable subgroup of $G$ normalized by $F$ and 
$$A_1N_1\nilrad(G)<S<M_2A_1N_1G_2.$$
Identifying $\M$ with $(S,g)$ where $g\in \Lc(S)$, we have
	 $$F=S\rtimes\Aut_\orth(S,g).$$
The group $S$ is in standard position in $G$.   
\item Up to conjugacy in $G$, every simply-transitive solvable subgroup $R$ of $G$ is contained in $F$ and thus is a modification of a conjugate of  $S$.
Moreover the standard modification procedure described in Remark~\ref{stdmodalg} carries $\rf$ to the Lie algebra of  a conjugate of  $S$.

\end{enumerate}
\end{prop}

See \cite{GordonWilson:TheFineStructureOfTransitiveRiemannianIsometryGroups} for the second bullet point in item (i), and see Propositions 1.7, Theorem 1.9 and Theorem 1.11 of \cite{GordonWilson:IsomGrpsOfRiemSolv} for the remaining statements of Proposiiton~\ref{stdpos}.

\begin{remark}\label{rem: unique comp solv} While some references speak of Iwasawa decompositions $H=KAN$ only in the setting of semisimple Lie groups of noncompact type, we are using this language in the more general setting of all semisimple Lie groups.  In this general setting, every simple normal compact subgroup of $H$ necessarily lies inside the subgroup $M$ defined in Proposition~\ref{stdpos}.    

The reason for our use of the subscripts in $A_1$ and $N_1$ is for convenience in what follows.

\end{remark}

In the remainder of this subsection, we restrict our attention to simply-connected completely solvable solvmanifolds.

\begin{prop}\label{lem.s1}  Let $(S,g)$ be a simply-connected completely solvable Riemannian solvmanifold and let $G=\Isom_0(S,g)$.   Then $S$ is in standard position.   In particular, $S$ and the isotropy group $L$ satisfy the conclusions of Proposition~\ref{stdpos}.   Moreover, we have:
\begin{enumerate}
\item $S=S_1\ltimes S_2$ where $S_1:=A_1N_1$ and $S_2:= S\cap G_2.$
\item Writing $N_2:=\nilrad(S_2)$, we have $\nilrad(S)=N_1\ltimes N_2$ and $N_2=\nilrad(G)$.
\item $G_2 =T\ltimes S_2$, where $T:=L\cap G_2$ is a torus that commutes with $G_1$.
\end{enumerate}
\end{prop}

We caution that the product $(L\cap G_1) T$ need not be all of $L$.  See, for example, the proof of Lemma~\ref{lem:Hermitian}.

\begin{proof} It is straightforward to see that completely solvable Lie groups remain unchanged during the standard modification algorithm described in Remark~\ref{stdmodalg}.  Thus $S$ is necessarily in standard position.

 (i). The fact that $\s$ is a completely solvable Lie algebra implies that its image $\pi(\s)$ under the projection $\pi:\g\to \g/\g_2\simeq \g_1$ is isomorphic to a completely solvable subalgebra of $\g_1$.   Combining this fact with  Proposition~\ref{stdpos}(ii), we have 
\begin{equation}\label{decs}(A_1N_1)\nilrad(G)<S<(A_1N_1)G_2\end{equation}
and (i) follows.

(ii) By Proposition~\ref{stdpos}(ii), we have $\nilrad(\g)<N_2$ and thus $\nilrad(\g)<\n_2$.  Since $[\g,\g_2]<\nilrad(\g)$, every subspace of $\g_2$ containing $\nilrad(\g)$ is a $\g$-ideal.  In particular, $\n_2$ is a nilpotent $\g$-ideal, and thus $\n_2<\nilrad(\g)$, so $N_2=\nilrad(G)$.   The fact that $\nilrad(S)=N_1\ltimes N_2$ is a straightforward consequence of (i).

 (iii)  Proposition ~\ref{stdpos}(i), statement (i) of the current proposition, and the fact that $G=LS$ together imply that   $G_2=(L\cap G_2) S_2$ and that $L\cap G_2$ commutes with $G_1$.  By (ii), $S_2$ is normal in $G$ so the product is semi-direct.   Letting $T=L\cap G_2$, then $T$ is compact and solvable, hence a torus.  
 
   \end{proof}

We continue to use the notation of Propositions~\ref{stdpos} and \ref{lem.s1}.

\begin{prop}\label{lem.s2}
 In addition to the hypotheses of Proposition~\ref{lem.s1}, suppose that $S$ can be written as a semi-direct product $S=AN$, where $N=\nilrad(S)$ and where $A$ is an abelian subgroup of $S$ satisfying the condition that $\ad_\s(\af)$ is reductive.  Then after possibly replacing $A$ by another subgroup of $S$ with the same properties, we have a decomposition 
$$A=A_1A_2 $$
where $A_1<S_1$ is defined as in Proposition~\ref{stdpos} and where $A_2$ is an abelian subgroup of $S_2$ satisfying:
	\begin{enumerate}
	\item $S_2=A_2\ltimes \nilrad(G)$.
	\item $A_2$ commutes with $G_1T$, where $T$ is the torus defined in Proposition~\ref{lem.s1}.     
	\end{enumerate}

\end{prop}

\begin{proof} The subalgebra $\af_1$ is $\ad$-reductive and thus lies in a maximal $\ad$-reductive subalgebra $\cf$ of $\s$.   By \cite[Corollary 4.2]{Mostow:FullyReducibleSubgrpsOfAlgGrps}, all maximal $\ad$-reductive subalgebras of $\s$ are conjugate.   One such is $\af+Z(\s)$, so all such are abelian.   Letting $\pi:\g\to\g_1$ be the projection with kernel $\g_2$, we have that $\pi(\cf)$ is an $\ad_{\s_1}$-reductive subalgebra containing $\af_1$ and thus must coincide with $\af_1$.   Hence $\cf=\af_1+(\cf\cap \s_2)$.   We have $Z(\s)< \cf\cap \s_2$.   Since $\s=\cf+\n$, we must have $\s_2=(\cf\cap \s_2)+\n_2$, with the intersection being $Z(\s)$.   A priori, $\cf\cap \s_2$ need not be a maximal $\ad_{\s_2}$-reductive subalgebra of $\s_2$ since $Z(\s_2)$ could properly contain $Z(\s)$.   However $(\cf\cap \s_2)+Z(\s_2)$ is a maximal $\ad_{\s_2}$-reductive subalgebra. In particular, all  maximal $\ad_{\s_2}$-reductive subalgebras are abelian and $\s_2$ is the sum of $\n_2$ and any such subalgebra. 

The subalgebra $\g_1+\uf$ of $\g$ is the Lie algebra of a fully reducible group of automorphisms of $\s_2$ and thus must leave some maximal $\ad_{\s_2}$-reductive subalgebra $\bfrak$ of $\s_2$ invariant (\cite[Corollary 5.1]{Mostow:FullyReducibleSubgrpsOfAlgGrps}).  Since $[\g,\s_2]<\n_2$ and, necessarily, $\bfrak\cap \n_2=Z(\s_2)$, we can decompose $\bfrak =Z(\s_2)\oplus\af_2$ where $\af_2$ commutes  with $\g_1+\uf$.  This is the desired subsapce $\af_2$.
\end{proof}

\subsection{Equivalence classes of solvable Lie groups.}

\begin{defin}\label{def.abstmod}
Let $R$ and $R'$ be simply-connected solvable Lie groups.  We will say that $R'$ is an \emph{abstract modification} of $R$ and that the Lie algebra $\rf'$ is an \emph{abstract modification} of $\rf$ if $R'$ is isomorphic to a modification of $R$ in $R\rtimes U$ for some compact connected subgroup $U$ of $\Aut(R)$.

\end{defin} 

In Definition~\ref{def.abstmod}, one can always choose $U$ to be a maximal compact connected subgroup of $\Aut(R)$ since if $U\subset U'<\Aut(R)$, where $U'$ is compact, then any modification in $R\times U$ can be viewed as a modification in $R\times U'$.   We emphasize that we are working in an abstract setting here so only require that $R'$ be isomorphic to the modification.

\begin{lemma}\label{abstmod conditions} Let $R$ and $R'$ be simply-connected solvable Lie groups.  The following are equivalent:
\begin{enumerate}
\item $R'$ is an abstract modification of $R$.
\item There exists a Lie group $G$ containing $R$ such that $R'$ is isomorphic to a modification of $R$ in $G$ in the sense of Definition~\ref{def.mod}.

\end{enumerate}

\end{lemma}

\begin{proof} 

The implication (i) $\implies$ (ii) is immediate.    

To see that (ii) $\implies$ (i), let $R$ and $R'$ be simply-connected solvable subgroups of a Lie group $G$ and suppose that $R'$ is a modification of $R$ in $G$ as in Definition~\ref{def.mod}.  Let  $\varphi:\rf\to\g$ be the associated modification map.   Since $\varphi(\rf)$ normalizes $\rf$, the adjoint action carries $\varphi(\rf)$ to a subalgebra $\uf$ of $\Der(\rf)$.  Thus $\rf'$ is isomorphic to a modification of $\rf$ in $\rf\rtimes \Der(\uf)$ and (i) follows. 
\end{proof}

\begin{defin}\label{def.scal} For $n\in \Z^+$, let $\Rcal_n$ denote the set of all (isomorphism classes) of $n$-dimensional simply-connected solvable Lie groups, and let $\Scal_n<\Rcal_n$ denote the subset consisting of all (isomorphism classes) of simply-connected completely solvable Lie groups.   We will abuse notation and view elements of $\Rcal_n$ and of $\Scal_n$ both as solvable Lie groups and as isomorphism classes of such groups.

\end{defin}

\begin{nota}\label{notajordan}	Given a Lie algebra $\g$ and $\delta\in\Der(\g)$, let $\delta=\delta_s+\delta_n$ be its Jordan decomposition, with $\delta_s$ semisimple and $\delta_n$ nilpotent.  By \cite[p. 203]{Mostow:FullyReducibleSubgrpsOfAlgGrps}, the fact that $\Der(\g)$ is the Lie algebra of an algebraic group implies that it contains the Jordan components of each of its elements, so $\delta_s$ and $\delta_n$ lie in $\Der(\g)$.   We further write $\delta_s=\delta_s^R + \delta_s^{iR}$, where $\delta_s^\R$ and $\delta_s^{i\R}$ are commuting semisimple endomorphisms with only real, respectively purely imaginary, eigenvalues.  It is easily verified that $\delta_s^\R$ and $\delta_s^{i\R}$ must also lie in $\Der(\g)$.
\end{nota}

\begin{lemma}\label{modprops2}\cite{Jablo:MaximalSymmetryAndUnimodularSolvmanifolds} Let $R$ be a simply-connected solvable Lie group.  Using Notation~\ref{notajordan}, define $\varphi: \rf \to \Der(\rf)$ by 
$$\varphi(X)=- \ad(X)_s^{i\R}.$$
Then:
\begin{enumerate}
\item  $\varphi$ is a modification map and 
$$\s:=\{X-\ad(X)_s^{i\R}:X\in \rf \}$$
 is a completely solvable normal modification of $\rf$ in $\rf\rtimes \Der(\rf)$.
\item Letting $U$ be a maximal compact subgroup of $\Aut(R)$ whose Lie algebra contains $\varphi(\rf)$, then the simply-connected completely solvable subgroup $S$ of $R\rtimes U$ with Lie algebra $\s$ is normal in  $R\rtimes U$.
\end{enumerate}
\end{lemma}

The proof  in \cite[pages 422-423]{Jablo:MaximalSymmetryAndUnimodularSolvmanifolds} is in the setting of unimodular solvable Lie groups but is valid without the assumption of unimodularity.  

\begin{thm}\label{thm: equiv classes} We use the notation of Definition~\ref{def.scal}. Define a map $\sigma:\Rcal_n\to\Scal_n$ by letting $\sigma(R)$ be the modification of $R$ defined in Lemma~\ref{modprops2} for each $R\in\Rcal$.  Define an equivalence relation on $\Rcal_n$ by the condition
 $$R\equiv R' \,\,\iff\,\,\sigma(R)\simeq \sigma(R').$$  Then:
	\begin{enumerate}
	\item Each equivalence class in $\Rcal_n$ contains exactly one (up to isomorphism) completely solvable Lie group $S$.  Moreover, every element in the equivalence class is an abstract modification of $S$.
	\item The following conditions are equivalent:
		\begin{enumerate}[label=(\alph*),ref=(\alph*)] 
		\item\label{first} $R\equiv R'$
		\item\label{second} There exist left-invariant metrics $g$ on $R$ 	and $g'$ on $R'$ such that $(R,g)$ is isometric to $(R',g')$.   
		\item\label{third} $R'$ can be obtained from $R$ by a succession of abstract modifications.
		\end{enumerate}
	\end{enumerate}
\end{thm}

\begin{proof}[Proof of Theorem~\ref{thm: equiv classes}]
Existence and uniqueness of $S$ are immediate since $\sigma(S)=S$ whenever $S$ is completely solvable.    The second statement in (i) follows from the fact that $\sigma(R)$ is a normal modification of $R$ as indicated in Lemma~\ref{modprops2}.

We prove statement (ii).   
 \ref{first}  $\implies$ \ref{second}:    Let $S\simeq\sigma(R)\simeq\sigma(R')$, let $U$ be a maximal connected compact subgroup of $\Aut(S)$, and choose $g_0\in \Lc(S)$ such that $g_0$ is $U$-invariant.   Viewing $R$ and $R'$ as modifications of $S$ in $S\rtimes U<\Isom(S,g_0)$, then $R$ and $R'$  are simply-transitive subgroups of $\Isom(S,g_0)$ and thus each inherits a Riemannian metric isometric to $g_0$.

\ref{second} $\implies$ \ref{third}:  Let $\M$ be the Riemannian solvamanifold $(R,g)$.  We may view both $R$ and $R'$ as simply-transitive solvable subgroups of $\Isom(\M)$.   By the last statement of Proposition~\ref{stdpos}, the standard modification algorithm carries $R$ to a subgroup $R''$ of $\Isom(\M)$ in standard position, and another modification carries $R''$ to a copy of $R'$.

\ref{third} $\implies$ \ref{first}:  It suffices to show that if $R'$ is an abstract modification of $R$, then $\sigma(R)=\sigma(R')$.   This assertion is proven in \cite{Jablo:MaximalSymmetryAndUnimodularSolvmanifolds} although the language of abstract modifications is not used.    We give a slightly different proof here.  We may view $R'$ as a modification of $R$ in $H:=R\rtimes U$ where $U$ is a compact connected subgroup of $\Aut(R)$.  Let $S=\sigma(R)$.   Lemma~\ref{modprops2} implies that $H=S\rtimes U$ and thus $R'$ is also a modification of $S$ in $H$.  (See Proposiiton~\ref{prop: modification picking up metric}.)   By Proposition~\ref{modprops}, this modification is normal and thus $S$ is a modification of $R'$.  By Lemma~\ref{abstmod conditions}, we have $\sigma(R')=S$.

\end{proof}

\section{Maximal, almost maximal, and infinitesimal maximal symmetry}

\begin{defin}\label{def:max sym} Let $H$ be a connected Lie group.  As before, we let $\Lc(H)$ denote the space of all left-invariant Riemannian metrics on $H$.   
	\begin{enumerate}
	\item A left-invariant Riemannian metric $g_0$ on $H$ is said to be \emph{maximally symmetric}, respectively \emph{almost maximally symmetric}, if for every $g\in \Lc(H)$, there exists $\Phi\in\Aut(H)$ such that $\Isom(H,g) < \Phi\,\Isom(H,g_0)\Phi^{-1}$, respectively $\Isom_0(H,g) < \Phi\,\Isom_0(H,g_0)\Phi^{-1}$.

	\item  We say that a left-invariant Riemannian metric $g_0$ on $H$ is \emph{infinitesimally maximally symmetric} if for every $g\in \Lc(H)$, the isometry algebra $\cisom(H, g)$ is abstractly isomorphic to a subalgebra of $\cisom(H,g_0)$. 
	\end{enumerate}
\end{defin}

\begin{remark}\label{rem.contrast}~
\begin{enumerate}
\item The condition $\Isom(H,g) < \Phi\,\Isom(H,g_0)\Phi^{-1}$ in Definition~\ref{def:max sym} is equivalent to the condition $\Isom(H,g)< \Isom(H,\Phi^*g_0)$.  The latter expression was used in the definition of maximally symmetric metric in \cite{GordonJablonski:EinsteinSolvmanifoldsHaveMaximalSymmetry}.
\item We emphasize that in the definitions of maximal and almost maximal symmetry, the inclusions are not merely abstract monomorphisms but rather inclusions of transformation groups of $H$.   We contrast this with the definition of infinitesimal maximal symmetry.    Here one has an abstract monomorphism $\cisom(H, g)\to \cisom(H, g_0)$ that need not respect the infinitesimal actions of these isometry algebras; moreover the monomorphism need not give rise to a Lie group homomorphism from $\Isom_0(H,g)$ to $\Isom_0(H,g_0)$.  See, e.g., 
Example~\ref{ex: ozeki}.
 
\end{enumerate}

\end{remark}

\subsection{Existence of metrics exhibiting the various types of maximal symmetry}

For arbitrary Lie groups, the question of existence of left-invariant metrics satisfying the various conditions in Defintion~\ref{def:max sym} is very difficult.     We are continuing to investigate this question, see e.g. \cite{EpsteinJablonski:MaximalSymmetryAndSolvmanifolds}.

\begin{prop}\label{prop.normal} Let $H$ be a connected Lie group.  
\begin{enumerate}
\item Suppose that every $g\in \Lc(H)$ satisfies $\Isom_0(H,g)< H\rtimes \Aut(H)$.  Then $H$ admits almost maximally symmetric left-invariant  metrics.   
Specifically, a metric $g\in \Lc(H)$ is almost maximally symmetric if and only if $g$ is invariant under some maximal compact subgroup of $\Aut_0(H)$.  

\item \label{prop.normal.item2} \cite{GordonJablonski:EinsteinSolvmanifoldsHaveMaximalSymmetry}   Suppose that every $g\in \Lc(H)$ satisfies $\Isom(H,g)< H\rtimes \Aut(H)$ and that $\Aut(H)$ has only finitely many components.   (The  latter condition is always satisfied if $H$ is simply-connected.)   Then $H$ admits maximally symmetric metrics, namely the metrics $g\in \Lc(H)$ that are invariant under some maximal compact subgroup of $\Aut(H)$.  
\end{enumerate}
	\end{prop}

Example~\ref{exa: torus} will show that the hypothesis on $\Aut(H)$ in statement (\ref{prop.normal.item2}) cannot be removed.

We omit the elementary proof of (i), which is based on the same ideas as the proof of (\ref{prop.normal.item2}) given in \cite{GordonJablonski:EinsteinSolvmanifoldsHaveMaximalSymmetry}

\begin{cor}\label{cor.pos}~
\begin{enumerate}
\item (See \cite{GordonJablonski:EinsteinSolvmanifoldsHaveMaximalSymmetry}.)  Every compact connected simple Lie group and every connected semisimple Lie group of noncompact type admit maximally symmetric left-invariant metrics.
\item Every connected completely solvable unimodular Lie group $H$ admits almost maximally symmetric left-invariant metrics.   Moreover \cite{GordonJablonski:EinsteinSolvmanifoldsHaveMaximalSymmetry}, if $H$ is simply-connected, then $H$ admits maximally symmetric left-invariant metrics.
\item   Every connected Lie group $H$ satisfying the conditions that some (hence every) semisimple Levi factor is of noncompact type and that $[\h,\h]=\h$ admits almost maximally symmetric metrics.  

\end{enumerate}

\end{cor}

\begin{proof}
  \cite{GordonWilson:IsomGrpsOfRiemSolv}, respectively \cite[Theorem 5.1]{Gordon:NilpRad}, show that the hypothesis of Proposition~\ref{prop.normal}(i) holds for completely solvable unimodular Lie groups, respectively for connected Lie groups satisfying the conditions of statement (iii).  The corollary follows.

\end{proof}

\begin{prop} Let $H$ be a connected Lie group.  If the solvable radical $\rad{H}$ is nilpotent and if the maximal normal compact subgroup of a Levi factor of $H$ commutes with $\rad{H}$, then $H$ admits infinitesimally maximally symmetric metrics.   In particular, every connected semisimple Lie group admits infinitesimally maximally symmetric metrics.

\end{prop}

\begin{proof} Let $K$ be a maximal connected compact subgroup of $\Aut(H)$ and let $g_0$ be a left-invariant metric on $H$ that is also $K$ invariant.  We show that $g_0$ is infinitesimally maximally symmetric.  Let $g\in \Lc(H)$.  By \cite[Theorem 4.1]{Gordon:NilpRad}, there exists a connected normal subgroup $H'$ of $\Isom_0(H,g)$ such that $H\simeq H'$ and $\Isom_0(H,g)=H'L$, where $L$ is the isotropy group.  In particular, $H'$ inherits a left-invariant metric that we again denote by $g$, and we have $L<\Aut_\orth(H', g')$ where $g'$ is the left-invariant metric on $H'$ induced by $g$.  By the conjugacy of maximal compact subgroups of $\Aut(H)$ and the fact that $H\simeq H'$, the isotropy group $L$ is isomorphic to a subgroup $K'$ of $K$ and 
  $$\cisom(H,g)\simeq \cisom(H',g) =\h'\rtimes \lf \simeq \h \rtimes \kf' <\cisom(H,g_0)$$
  and the proposition follows. \end{proof}

\subsection{Comparing the three types of maximal symmetry}

Observe that
$$\mbox{Max\, symmetry} \implies  \mbox{Almost \, max\, symmetry} \implies \mbox{Inf \,max\, symmetry}.$$

Neither converse implication holds in full generality.    There are two ways that the converse of the first implication can fail.   
\begin{enumerate} \item\label{type one} A Lie group $H$ that admits metrics of maximal symmetry may admit a strictly larger collection of almost maximally symmetric metrics.  
\item\label{type two}  A Lie group may admit almost maximally symmetric metrics but no maximally symmetric ones
\end{enumerate}
Examples of the first type are Lie groups $H$ that satisfy the hypothesis of Proposition~\ref{prop.normal}(i) and whose automorphism group contains more than one but only finitely many components.  Let $L$ be a maximal compact subgroup of $\Aut(H)$ and $L_0$ its identity component.   Left-invariant metrics that are $L_0$-invariant but not $L$-invariant are almost maximally symmetric but not maximally symmetric.     The following example is of the second type.

 \begin{example}\label{exa: torus}
Consider the torus $T=\R^2/\Z^2$.   Every $g\in \Lc(T)$ is flat and satisfies $\Isom_0(T,g)=T$.  Thus every left-invariant metric is almost maximally symmetric.   However, it is easy to see that $T$ does not admit any maximally symmetric left-invariant metric.

 \end{example}
 
 We next compare infinitesimal maximal symmetry with almost maximal symmetry.   Behavior analogous to that in~(\ref{type one}) cannot occur:
 
 \begin{prop}\label{no type one}
 Let $H$ be a connected Lie group that admits an almost maximally symmetric left-invariant metric $g_0$.  Then every infinitesimally maximally symmetric metric on $H$ is almost maximally symmetric.
 \end{prop}
 
 \begin{proof}
 Suppose that $g\in\Lc(H)$ is also infinitesimally maximally symmetric.   Then each of $\cisom(H,g)$ and $\cisom(H,g_0)$ is isomorphic to a subalgebra of the other, so 
 	$$\cisom(H,g)\simeq \cisom(H,g_0).$$   Since also $\Isom_0(H,g)<\Phi\,\Isom_0(H,g_0)\Phi^{-1}$ for some $\Phi\in \Aut(H)$, we must have 
 	$$\Isom_0(H,g)=\Phi\,\Isom_0(H,g_0)\Phi^{-1}$$ 
and thus $g$ is also almost maximally symmetric. 

  \end{proof}
 
The following example shows that some Lie groups admit infinitesimally maximally symmetric but no almost maximally symmetric ones.   Later in this subsection, we will see a similar phenomenon exhibited by some completely solvable simply-connected Lie groups.

\begin{example}\label{ex: ozeki} 
Let $H$ be a semisimple compact connected Lie group and let $g_0$ be a bi-invariant Riemannian metric on $H$.  In the notation of \ref{isomgrps}, we have $\cisom(H,g_0)=\lambda_*^{g_0}(\h)\times \rho_*^{g_0}(H)\simeq \h\times \h$, where $\lambda^{g_0}$ is the group of left translations as in Notation~\ref{isomgrps} and $\rho^{g_0}(H)$ is the analogously defined group of right translations.  Ozeki \cite{Oze77} showed that for every $g\in \Lc(H)$, there exists an injective homomorphism  $\alpha: \cisom(H,g)\to \h\times \h$.  It follows that $g_0$ is infinitesimally maximally symmetric.   However, Ozeki gave examples of semisimple compact connected Lie groups $H$ that admit left-invariant metrics $g$ for which the group of left translations $\lambda^g(H)$ is \emph{not} normal in $\Isom_0(H,g)$.  In this case, $\Isom_0(H,g)\not\subset \Phi\Isom(H,g_0)\Phi^{-1}$ for any $\Phi\in \Aut(H)$, since conjugation by $\Phi$  normalizes $\lambda^g(H)$.  Therefore $g_0$ is not almost maximally symmetric.  Moreover, by Proposition~\ref{no type one}, $H$ cannot admit any almost maximally symmetric left-invariant metric.    Aside:   For simple compact Lie groups, this type of behavior does not occur and the bi-invariant metric is maximally symmetric.
  \end{example}
  
  \begin{remark}\label{comp solv alpha}~\begin{enumerate}
  \item  Suppose that $g_0$ is an infinitesimally maximally symmetric left-invariant metric on a Lie group $H$.   By Definition~\ref{def:max sym}, for every $g\in \Lc(H)$, there exists a monomorphism $\alpha:\cisom(H,g)\to\cisom(H,g_0))$.   Observe that a necessary condition for $g_0$ to be almost maximally symmetric is that the monomorphism $\alpha$ can be chosen to satisfy
  \begin{equation}\label{eq: strong alpha} \alpha(\lambda^g_*(\h))=\lambda^{g_0}_*(\h).\end{equation}
This condition clearly fails for the metric $g$ in Example~\ref{ex: ozeki}.

\item In contrast to the behavior in Example \ref{ex: ozeki}, if $g_0$ is an infinitesimally maximally symmetric metric on a simply-connected completely solvable Lie group $S$, then for any $g\in\Lc(S)$, the monomorphism $\alpha:\cisom(S,g)\to \cisom(S,g_0)$ can be chosen to satisfy~\eqref{eq: strong alpha}.    Indeed start with any choice of $\alpha$.  
Since  $\alpha(\lambda^g_*(\s))$ is a completely solvable subalgebra of $\cisom(S,g_0)$ isomorphic to $\s$, Proposition~\ref{lem.s1} and the uniqueness up to conjugacy of solvable subgroups in standard position in $\Isom(S,g_0)$ (see Proposition~\ref{stdpos}(iii)) enables us to replace $\alpha$ by its composition with a suitable automorphism of $\cisom(S,g_0)$ so that Equation~\ref{eq: strong alpha} holds.
\end{enumerate}
\end{remark}

In spite of Remark~\ref{comp solv alpha}, we will see in Example~\ref{type II example} that there exist simply-connected completely solvable Lie groups that admit infinitesimally maximally symmetric metrics but no almost maximally symmetric ones.    We first show, however, that this phenomenon can only occur if $S$ has non-trivial center.

\begin{prop}\label{comp solv inf implies almost}
Let $S$ be a simply-connected completely solvable Lie group.    If $\s$ has trivial center, then every infinitesimally maximally symmetric left-invariant Riemannian metric on $S$ is almost maximally symmetric.

\end{prop}

\begin{lemma}\label{effect kernel} We use the notation of Definition~\ref{def:max sym}.
Let $H$ be a simply-connected Lie group, let $g, g'\in \Lc(H)$, and let $L$ and $L'$ be the isotropy subgroups of $\Isom_0(H,g)$ and $\Isom_0(H,g')$, respectively, at the identity $e\in R$.   Suppose that there exists an injective homomorphism $\alpha: \cisom(H,g')\to \cisom(H,g)$ satisfying 
\begin{enumerate}
\item $\alpha(\lambda_*^{g'}(\h))=\lambda_*^{g}(\h)$, where $\lambda^g$ and $\lambda^{g'}$ are defined as in Notation~\ref{isomgrps};
\item $\alpha(\lf')<\lf$, where $\lf'$ and $\lf$ are the Lie algebras of $L'$ and $L$, respectively.
\end{enumerate}.  
Then there exists an automorphism $\Phi$ of $H$ such that $\Isom_0(H,g')  < \Phi\Isom(H,g)\Phi^{-1}$.
\end{lemma}

\begin{proof}[Proof of Lemma~\ref{effect kernel}]
Write $G=\Isom_0(H,g)$ and $G'=\Isom_0(H,g')$, let $\tG$ and $\tG'$ be the simply-connected covers of $G$ and $G'$ and let $\widetilde{L}<\tG$ and $\widetilde{L'}<\tG'$ be the connected subgroups corresponding to $\lf$ and $\lf'$.  The monomorphism $\alpha$ lifts to a monomorphism $\widetilde{\alpha}: \tG'\to\tG$ carrying $\tL'$ into $\tL$.   We have $\tG=\tilde{\lambda}^{g'}(H)\tL$ and $\tG'=\tilde{\lambda}^{g}(H)\tL'$, and each of the intersections $H\cap L$, $H\cap \tL$, etc. are trivial.   Thus $\widetilde{\alpha}$ induces a diffeomorphism $\tG'/\tL'\to \tG/\tL$ intertwining the left action $\widetilde{\rho'}$ of $\tG'$ on $\tG'/\tL'$ with the left action $\widetilde{\rho}\circ\widetilde{\alpha}$ of $\tG'$ on $\tG/\tL$.   In particular, $\widetilde{\alpha}$ maps the effective kernel $D'<\tL'$ of $\widetilde{\rho'}$ into the effective kernel $D$ of $\widetilde{\rho}$.  Since $G'=\tG'/D'$ and $G=\tG/D$, the map $\widetilde{\alpha}$ descends to a monomorphism $\hat{\alpha}:G'\to G$ with $\hat{\alpha}_*=\alpha$.    The map $\hat{\alpha}$ induces a diffeomorphism $\tau:G'/L'\to G/L$ intertwining the left actions of $G'$ and $G$, and in particular, intertwining the simply-transitive actions of $\lambda^{g'}(H)$ and $\lambda^g(H)$.  The diffeomorphism $\tau$ corresponds to the automorphism $\Phi$ of $H$ defined by the condition $$\hat{\alpha}\circ \lambda^{g'}= \lambda^g\circ \Phi.$$   The pullback to $G'/L'$ by $\tau$ of the Riemannian metric on $G/L$ defined by $g$ corresponds to the metric $\Phi^*g$ on $H$.    Hence we have $\Isom_0(H,g')=G'<\Isom(H,\Phi^*g)=\Phi\Isom(H,g)\Phi^{-1}$.   
\end{proof}

\begin{proof}[Proof of Proposition~\ref{comp solv inf implies almost}]  
Suppose $g$ is an infinitesimally maximally symmetric metric on $S$.  Let $g'\in \Lc(S)$.   Remark~\ref{comp solv alpha} guarantees the existence of a  monomorphism $\alpha:\cisom(S,g')\to \cisom(S,g)$ satisfying condition (i) of the lemma (with $S$ playing the role of $H$ in the lemma).   Write $G=\Isom_0(S,g)$ and $G'=\Isom_0(S,g')$.  For notational simplicity, write $S_g:=\lambda^g(S)<G$ and similarly for $S_{g'}<G'$.  Use the analogous notation for their Lie algebras.  Thus condition (i) says  $\alpha(\s_{g'})=\s_g$.  

We next address condition (ii).  Let $\lf$ and $\lf'$ be the isotropy subalgebras of $\g$ and $\g'$ at the base point.     Write $\h=\alpha(\g')<\g$ and let $H$ be the corresponding connected subgroup of $G$.   We have $H=(H\cap L)S_g$.    The hypotheses on $S$ imply that the only compactly embedded subalgebra of $\s$ is $\{0\}$ and  thus the same holds for its isomorphic copies $\s_{g'}$ and $\s_g$.  Hence both $\lf\cap \h$ and $\alpha(\lf')$ are maximal compactly embedded subalgebras of $\h$.   By the conjugacy of maximal compactly embedded subalgebras, there exists $a\in H$ such that $\Adop(a)(\lf\cap \h) =\alpha(\lf')$.  Since $\Adop_H(H\cap L)$ normalizes $\lf\cap \h$, we may choose $a$ to lie in $S_g$.   Extending the inner automorphism $I_a$ to $G$, we thus have $\alpha(\lf')<\Adop_G(a)(\lf)$.  Thus, after replacing $g$ by $I_a^*(g)$, we obtain condition (ii).   (This change has not affected condition (i) since $a\in S_g$.) The proposition now follows from the lemma.

\end{proof}

In contrast to Proposition~\ref{comp solv inf implies almost}, we next give examples of completely solvable simply-connected Lie groups $S$ that admit infinitesimally maximally symmetric metrics but no almost maximally symmetric ones.     Our examples are of two types:

\begin{itemize}
\item[Type I.]   There exist two infinitesimally maximally symmetric metrics $g_1$ and $g_2$ on $S$ (necessarily with isomorphic isometry algebras) such that $\Isom_0(S,g_1)$ is not isomorphic to $\Isom_0(S,g_2)$.   This condition implies that $S$ cannot admit an almost maximally symmetric metric.
\item[Type II.]  Up to automorphism and rescaling, $S$ admits a unique infinitesimally maximally symmetric metric $g$, but $g$ is not almost maximally symmetric.
\end{itemize}

Both constructions use the following:

\begin{lemma}\label{lem:Hermitian}
Let $\g$ be a Lie algebra.  Suppose the Levi decomposition $\g=\g_1+\g_2$ and Iwasawa decomposition $\g_1=\kf+\s_1$ satisfy:
\begin{itemize}
\item $\g_1$ is of noncompact type and $\kf$ has non-trivial center.
\item $\g_2$ is completely solvable.
\item The center $Z(\g)$ is non-trivial.
\end{itemize}
Let $S$ be the simply-connected Lie group with Lie algebra $\s_1\ltimes \g_2$.  Then there exist connected Lie groups $G$ and $G'$, both with Lie algebra $\g$, and compact subgroups $K$ of $G$ and $L$ of $G'$ such that:

\begin{enumerate}
\item $G$ is a linear Lie group but $G'$ is not.
\item $S$ acts simply transitively on both $G/K$ and $G'/L$.

\end{enumerate}
\end{lemma}

\begin{proof}For notational simplicity, we assume that $\g_1$ is simple.  The proof is easily modified for the general case. Let $\tG$ be the simply-connected Lie group with Lie algebra $G$.     Then $\tG=\tK S$ and $\tG_1=\tK S_1$, where $S$ is defined as in the statement of the lemma and $\tK$ and $S_1$ are the connected subgroups of $\tG$ with Lie algebras $\kf$ and $\s_1$.   Setting $D=\tG_1\cap Z(\tG)$, we have 
$$\Z\simeq D <Z(\tK)\simeq \R.$$   

Let $G=D\backslash\tG$ and $K=D\backslash \tK <G$.   Then $G$ is a linear Lie group and $S$ acts simply transitively on $G/K$.

We next define $G'$.   Let $Z(\tG)_0$ be the identity component of the center of $\tG$.  Since $\tG$ is simply-connected, we have $Z(\tG_0)\simeq \R^m$ for some $m\geq 1$.    Choose a one-dimensional subspace $W$ of $Z(\tG)$ and an isomorphism $\mu:\tK\to W$.   Let
$$W':=\{(a,\mu(a)): a\in Z(\tK)\} \,\,\,\mbox{and}\,\,\,D':=\{(z, \mu(z)): z\in D\}.$$
Set $G'=D'\backslash \tG$.   Let $K'$ be the connected subgroup of $G'$ with Lie algebra $\kf$ and $K'_1$ the semisimple part of $K'$.    The center of $K'$ is non-compact and $G'$ is not a linear Lie group.   However, letting 
$$L=K' \times (D'\backslash W'),$$
then $L$ is compact and $S$ acts simply transitively on $G'/L$.

\end{proof}

\begin{example}\label{type I example}[Type I.] Let $\g$ satisfy the hypotheses of Lemma~\ref{lem:Hermitian} and assume in addition that $\g_2=Z(\g)\simeq \R^m$ for some $m\geq 1$.     In the notation of the lemma, any Riemannian metric $g$ on $S$ induced by a left-invariant metric on $G/K$ is a solvsoliton.   (Indeed every such metric gives $S$ the structure of a Riemannian symmetric space.)  It will follow from Theorem~\ref{main} that $g$ is infinitesimally maximally symmetric.    However, since $S$ also acts simply transitively on $G'/L$, we can also give $S$ a left-invariant Riemannian metric $g'$ such that $\Isom_0(S,g')=G'$.   Since $G'$ is not isomorphic to any subgroup of $G$,  the Lie group $S$ cannot admit any almost maximally symmetric metric.
\end{example}

\begin{example}\label{type II example}[Type II.] 
Let $S=S^*\times \R$ where $S^*$ is the Iwasawa subgroup of $\SL(3,\R)$ given by the upper triangular matrices, and let $g$ be a metric on $S$ giving it the structure of a Riemannian symmetric space with $\Isom_0(S,g)= \SL(3,\R)\times \R$.   As in the previous example, it follows from Theorem~\ref{main} that $g$ is infinitesimally maximally symmetric.   In contrast to the previous example, up to isometry and rescaling,  $g$ is the unique left-invariant metric with isometry algebra $\slf(3,\R)\oplus \R$ and thus the unique infinitesimally  maximally symmetric metric.   

To see that $g$ is not almost maximally symmetric, consider the parabolic subalgebra $\h$ of $\slf(3,\R)$ given by all matrices in $\slf(3,\R)$ for which the first two entries of the bottom row are zero.   We have $\h=\h_1\ltimes \h_2$ with $\h_1\simeq \slf(2,\R)$ and
$$\h_2=\begin{bmatrix} a&0&b\\0&a&c\\0&0&-2a  \end{bmatrix}.$$  
Let $\g=\h\oplus\R$.   Then $\g$ satisfies the hypothesis of Lemma~\ref{lem:Hermitian} with $\g_1=\slf(2,\R)$ and $\g_2=\h_2\oplus\R$.  Note that the Lie group $S$ defined in the lemma coincides with the solvable Lie group $S$ defined at the beginning of this example.   By the lemma, there exists a left-invariant Riemannian metric $g'$ on $S$ such that $\Isom_0(S,g')$ is a non-linear Lie group $G'$ with Lie algebra $\g$.     Since $\Isom_0(S,g)$ is a linear Lie group, $\Isom_0(S,g')$ cannot be isomorphic to any subgroup of $\Isom_0(S,g)$ and thus $g$ is not almost maximally symmetric.
\end{example}

\subsection{Infinitesimal maximal symmetry and equivalence classes of solvable Lie groups.}

\begin{thm}\label{thm.R admits max sym} 
Let $R$ be a simply-connected solvable Lie group and let $S$ be the unique completely solvable Lie group in the equivalence class of $R$ (as defined in Theorem~\ref{thm: equiv classes}).  Suppose that $S$ admits an infinitesimally maximally symmetric left-invariant metric $g_0$   Then:
\begin{enumerate}
\item $R$ admits an infinitesimally maximally symmetric left-invariant Riemannian metric $g_0'$ isometric to $g_0$. 
\item  If, moreover, $g_0$ is almost maximally symmetric, then for every $g\in \Lc(R)$, there exists a left-invariant metric $g_0''\in \Lc(\R)$ isometric to $g_0'$ such that $\Isom_0(R,g)\subset \Isom_0(R, g_0'')$.
\end{enumerate}
 \end{thm}

Note that the conclusion of the second item is weaker than almost maximal symmetry of $g_0'$, as the metric $g_0''$ is not required to be the pullback of $g_0'$ by an automorphism of $R$.   Example~\ref{R not admit almost max} below shows that the stronger condition of almost maximal symmetry does not always hold.

The key step in proving the theorem is the following proposition:

\begin{prop}\label{prop:  reduction to compl solv solvsol} Let $R$ be a simply-connected solvable Lie group and let $S$ be the completely solvable Lie group in the equivalence class of $R$.    Let $g\in \Lc(\R)$.   Then there exists $g'\in \Lc(\R)$ such that $\Isom_0(R,g)<\Isom_0(R,g')$ and such that $g'$ is isometric to a left-invariant Riemannian metric on $S$. 
\end{prop}  

We first assume Proposition~\ref{prop:  reduction to compl solv solvsol} and prove the theorem.

\begin{proof}[Proof of Theorem~\ref{thm.R admits max sym}]
Note that $\Aut_\orth(S,g_0)$ must be a maximal connected compact subgroup of $\Aut(S)$.    (Otherwise, let $H$ be a connected compact subgroup of $\Aut(S)$ properly containing $\Aut_\orth(S,g)$.  Then the normalizer of $\s$ in $\cisom(S,g')$ has larger dimension than the normalizer of $\s$ in $\cisom(S,g)$, contradicting infinitesimal maximal symmetry of $g$.) 
The proof of the implication (a) $\implies$ (b) in Theorem~\ref{thm: equiv classes} shows  that $R$ admits a left-invariant Riemannian metric $g_0'$ isometric to $g_0$.  The fact that $g_0'$ satisfies the conclusions of the theorem is a straightforward consequence  of Proposition~\ref{prop:  reduction to compl solv solvsol}.

\end{proof}

To prove Proposition~\ref{prop:  reduction to compl solv solvsol}, we must take a careful look at the structure of the isometry group for a solvmanifold.  Let $\M=(R,g)$ and let $G=\Isom_0(\M)$.    We use the notation and results of Proposition~\ref{stdpos}, except that we write $S^*$ for the subgroup of $G$ in standard position in order to avoid confusion with the group $S$ in the statement of the proposition.   In particular, the isotropy group $L$ at a point of $\M$, the solvable group $S^*$, the Levi decomposition $G=G_1G_2$ and Iwasawa decomposition $G_1=KA_1N_1$ satisfy all the conditions in Proposition~\ref{stdpos} with $S^*$  playing the role of $S$.    

\begin{nota}\label{autl} Let $\Aut^L(G)$ denote the identity component of the subgroup of $\Aut(G)$ consisting of those automorphisms that fix $L$ pointwise.  (Note that such automorphisms commute with the inner automorphisms $\Inn_G(L)$.)   

Let $\Aut^L(\g)$ be the subgroup of $\Aut(\g)$ consisting of all automorphisms that fix $\mathfrak l = \mathrm{ Lie }~L$  pointwise; as above, these commute with $\Ad_G(L)<\Aut(G)$.  Let $\Der^L(\g)$ be the Lie algebra of $\Aut^L(G)$.
\end{nota}

\noindent\underline{Strategy of the proof of Propition~\ref{prop:  reduction to compl solv solvsol}.}   Let $H$ be a compact subgroup of $\Aut^L(G)$.   Setting $G'=G\rtimes H$ and $L'=L\rtimes H$, observe that the subgroup $R<G<G'$ acts simply transitively on $G'/L'$.     We will show that $H$ can be chosen in such a way that $G'$ contains an isomorphic copy of $S$ that also acts simply transitively on $G'/L'$.   Any left-invariant Riemannian metric on $G'/L'$ induces isometric left-invariant metrics $g'$ on $R$ and $S$.  Since $G<\Isom_0(R,g')$, the proposition will follow.

\begin{lemma}\label{D}~
\begin{enumerate}
\item The map $\Phi\to\Phi_*$ restricts to  an isomorphism from $\Aut^L(G)$ to the subgroup $\Aut^L(\g)$ of $\Aut(\g)$.  
\item  $\Aut^L(\g)$ is an algebraic subgroup of $\Aut(\g)$ whose Lie algebra is $\Der^L(\g)$, the space of all derivations that vanish on $\lf$ and thus that commute with $\ad_\g(\lf)$.   
\item $\Der^L(\g)\cap \ad(\g)=\ad_\g(C_\g(\lf))$.
\item $\ad_\g(C_\g(\lf))$ is a solvable ideal of $\Der^L(\g)$.
\item In the language of Notaion~\ref{notajordan}, for  $\delta\in \Der^L(\g)$, the Jordan components $\delta_n$, $\delta_s^\R$ and $\delta_s^{i\R}$ all lie in $\Der^L(\g)$.

\end{enumerate}
\end{lemma}

\begin{proof}(i) Let $\pi:\tG\to G$ be the universal cover of $G$, and let $D$ be the kernel of $\pi$ (a discrete central subgroup of $\tG$).  Let $M=G/L$ and let $\rho$ be the isometry action of $G$ on $\M$.  The action $\rho$ is effective.   $\tG$ acts on $M$ via $\rho\circ\pi$ with effective kernel $D$.   Necessarily $D$ lies in the isotropy subgroup $\tL$.   Since $\tG$ is connected and $\M$ is simply-connected, the isotropy subgroup $\tL$ is connected and thus must be the connected subgroup of $\tG$ with Lie algebra $\lf$.   

Since $\tG$ is simply-connected and $\tL$ is connected, we have $\Aut^\tL(\tG)\simeq \Aut^\tL(\g)$, where $\Aut^\tL(\g)$ is the subgroup of $\Aut(\g)$ consisting of all automorphisms that fix $\lf$ pointwise and commute with $\Ad_{\tG}(\tL)$.  Since $D$ is central and $L=\tL/D$, we have $\Ad_G(L)=\Ad_{\tG}(\tL):\g\to\g$.     Moreover, since $D<\tL$, every element of $\Aut^\tL(\tG)$ fixes $D$ and thus descends to an element of $\Aut^L(G)$.    The first statement of the lemma follows.

(ii) Since $\Aut(\g)$ is algebraic, the subgroup of $\Aut^L(\g)$ that stabilizes $\lf$ is also algebraic.

(iii) is immediate from the definition of $\Der^L(\g)$.  

(iv) Since $\ad(\g)$ is an ideal in $\Der(\g)$, (iv) follows from (iii) and the fact that $C_\g(\lf)$ lies in the solvable algebra $Z(\kf)+\g_2$. 

(v) follows from (ii) and the result of Mostow cited in Notation~\ref{notajordan}.

\end{proof}

 \begin{lemma}\label{centralizer} Let $\cf$ be the orthogonal complement of $\g_1+\lf$ in $\g$ relative to the Killing form.    In the notation   above, we have
 $$\nilrad(\g)<\cf<\s^*\cap\g_2 \mbox{\,\,\,and\,\,\,}(\g_1+\lf)\cap\cf<Z(\g).$$

\end{lemma}

\begin{proof}
The fact that $\cf<\s^*$ follows from Proposition~\ref{stdpos} since $S^*$ is in standard position.  The other assertions in (i) arise from elementary properties of the Killing form.

\end{proof}

\begin{lemma}\label{bfrak} There exists a complement $\bfrak$ of $\nilrad(\g)$ in $\cf$ such that:
\begin{enumerate}
\item $[\bfrak, \g_1+\lf]=0$.    In particular, $\bfrak<C_\g(\lf)$.   
\item $\{\ad_\g(X)_s: X\in\bfrak\}$ forms an abelian subspace of $\Der^L(\g)$, necessarily commuting with $\ad_\g(\g_1+\lf)$.
\end{enumerate}

\begin{remark}
While $\g$ is not necessarily an algebraic Lie algebra (i.e., the Lie algebra of an algebraic group), this lemma is inspired by the algebraic Lie algebra setting.    For algebraic Lie algebras $\uf$, one always has an abelian complement $\bfrak$ of $\nilrad(\uf)$ in $\uf_2$ that commutes with a given semisimple Levi factor and that satisfies the stronger conditions that the elements of $\ad_\uf\,\bfrak$ are semisimple.   \end{remark}

\end{lemma}

\begin{proof} Let $k=\dim(\cf)-\dim(\nilrad(\g))$.   We show by induction that for every $j=1,\dots, k$, there exists a $j$-dimensional subspace $\bfrak_j$ of $\cf$ disjoint from $\nilrad(\g)$ satisfying (i) and (ii) (with $\bfrak_j$ playing the role of $\bfrak$).  We can then take $\bfrak=\bfrak_k$.

Begin with any complement $\bfrak_0$ of $\nilrad(\g)$ in $\cf$ that commutes with $\g_1+\lf$.   Such a complement exists since $\ad_\g(\g_1+\lf)$ is reductive and maps $\cf$ into $\nilrad(\g)$.  Choose $X_1\in \bfrak_0$ and set $\bfrak_1=\spa(X_1)$.

Next suppose $\bfrak_j$ has been defined satisfying (i) and (ii).   Observe that $\vf:=\adg(\g_1+\lf)+\{\adg(X)_s:X\in\bfrak_j\}$ is a reductive subalgebra of $\Der^L(\g)$ that maps $\cf$ (in fact all of $\g_2$) into $\nilrad(\g)$ and annihilates $\g_1+\lf$.  Thus we may choose a complement $\bfrak'_j$ of $\bfrak_j +\nilrad(\g)$ in $\cf$ that is annihilated by $\vf$.  Choose $0\neq X_{j+1}\in\bfrak'_j$ and let $\bfrak_{j+1}=\bfrak_j +\R X_{j+1}$.   

\end{proof}

Let $\uf$ be the compactly embedded subalgebra of $\g$ given by
$$\uf= \kf+\lf +Z(\g).$$
Observe that $[\uf,\uf]=\kf_1<\lf$.   Choose a complement $\uf_0$ of $\lf+Z(\g)$ with $\uf_0< Z(\lf +\kf)$.    We then have vector space direct sums
\begin{equation}\label{vsds}\g=\lf +\s_1 +\cf+\uf_0=\lf+\s_1+\nilrad(\g)+\bfrak+\uf_0\end{equation}
and $ad_\g(\uf_0)<\Der^L(\g)$.  
Define 
$$\w=\adg(\uf_0) +\{\adg(X)_s^{iR}: X\in\bfrak\}<\Der^L(\g).$$
Observe that $\w$ is an abelian compactly embedded subalgebra of $\Der^L(\g)$ all of whose elements act fully reducibly with purely imaginary eigenvalues on $\g$.   Thus we may choose a (maximal if desired) compact connected subgroup $H$ of $\Aut^L(G)$ whose Lie algebra contains $\w$.    Set 
$$G'=G\rtimes H \mbox{\,\,\,and\,\,\,} L'=L\rtimes H.$$
We have $G'=L'R$ with $R\cap L'$ trivial.  Thus we may choose a Riemannian metric $g'$ on $G'/L'$, equivalently on $R$, so that $G'$ acts isometrically.   

\begin{nota} Write elements of $\g'$ as pairs $(X,\delta)$ with $X\in \g$ and $\delta\in \h$.

\end{nota} 

\begin{lemma}\label{finallemma} Let $$\w'=\{(X,-\ad_\g(X)_s^{iR}): X\in \bfrak\} + \{(X,-\ad_\g(X)): X\in\uf_0\}$$ and let $$\s=(\s_1+ \nilrad(\g))\times \{0\} + \w'.$$
Then 
\begin{enumerate}
\item $\g' =\lf' +\s$ and $\lf'\cap\s=\{0\}$.
\item $\s$ is a completely solvable Lie algebra.
\item The connected subgroup $S$ of $G'$ with Lie algebra $\s$ acts simply transitively on $G'/L'$.
\end{enumerate}\end{lemma}

 \begin{proof}  
(i) follows from the facts that Equation~\ref{vsds} is a direct sum and that 
$\bfrak +\uf_0< \lf'+\w'$.   
 
(ii) We first show that $\s$ is a solvable subalgebra.   By Lemma~\ref{bfrak}, $\w'$ commutes with $\s_1\times \{0\}$ and it normalizes $\nilrad(\g)\times \{0\}$, so we need only show that $[\w',\w']_{\g'} <\nilrad(\g)  \times \{0\}$.  (We are writing $[\cdot,\cdot]_{\g'}$ for the bracket in $\g'$.) Let $U=(X,-\adg(X)_s^{iR})$ and $V=(Y,-\adg(Y)_s^{iR})$ be in $\w'$.       (Aside:  If $X\in \uf_0$, then $\adg(X)_s^{iR}=\adg(X)$.)
Then $$[U,V]_{\g'}= ([X,Y] -\adg(X)_s^{iR}(Y) +\adg(Y)_s^{iR}(X))\times \{0\}.$$
Now $\adg(X)_s^{iR}$ and $\adg(X)_s^{iR}$ restrict to derivations of the solvable Lie algebra $\g_2$ and thus must map it into the nilradical.  Thus $[U,V]_{\g'}\in\nilrad(\g)\times\{0\}$ and $\s$ is a solvable Lie algebra, and one easily checks that it is completely solvable.

(iii) follows from (i) and the fact that $G'/L'$ is simply-connected.

\end{proof}

By Lemma~\ref{finallemma}, $S$ admits a left-invariant metric isometric to the left-invariant metric $g'$ on $R$, so Theorem~\ref{thm: equiv classes} guarantees that $S$ is the completely solvable Lie group in the equivalence class of $R$, thus completing the proof of Proposition~\ref{prop:  reduction to compl solv solvsol}.

The following example shows that the conclusion of the second item in Theorem~\ref{thm.R admits max sym}  cannot be strengthened to almost maximal symmetry of $g_0'$.   

\begin{example}\label{R not admit almost max}
We first construct a completely solvable simply-connected Lie group $S$ that admits an almost maximally symmetric left-invariant metric.  We then give a modification $R$ of $S$ for which the resulting infinitesimally maximally symmetric metric on $R$ (whose existence is guaranteed by Theorem~\ref{thm.R admits max sym}) is not almost maximally symmetric.   By Proposition~\ref{no type one}, it will follow that $R$ does not admit any almost maximally symmetric metric.

Define the Lie algebra $\s$ to be $\af \ltimes \R^6$, with $\dim(\af)=1$, where the action of $\af$ on $\R^6$ is given as follows:   Fix $A\neq 0$ in $\af$, let $\{E_1,\dots, E_6\}$ be a basis of $\R^6$ and set 
$$\ad(A)_{|\R^6}=\begin{bmatrix}\rm{I}_2 & & \\&-\rm{I}_2 &\\& & 0_2\end{bmatrix}$$
relative to the basis of $\R^6$ above.    Here $\rm{I}_2$ and $0_2$ are the $2\times 2$ identity matrix and the $2\times 2$ zero matrix, respectively.   The Lie algebra $\s$ is completely solvable and unimodular with nilradical $\n=\R^6$.   Let $S$ be the simply-connected Lie group with Lie algebra $\s$.  

Define an action $\rho$ of the group $\SO(2)\times\SO(2)\times \SO(2)$ on $\s$ fixing $\af$ pointwise, leaving $\R^6$ invariant and acting on $\R^6$ as 
$$\begin{bmatrix}\SO(2)&&\\&\SO(2)&\\&&\SO(2)\end{bmatrix}$$
relative to the basis above.  Then this action is by automorphisms of $\s$ and, letting $H=\rm{Image}(\rho)$, one easily checks that  
$H$ is a maximal compact subgroup of $\Aut_0(\s)$.     We use the same notation $H$ for the corresponding maximal compact subgroup of $\Aut_0(S)$.  

The left-invariant metric $g_0$ on $S$ defined by declaring the basis 
$$\B_\s:=\{A, E_1,\dots, E_6\}$$ to be orthonormal is $H$-invariant.   Since $S$ is completely solvable and unimodular, Corollary~\ref{cor.pos} and its proof impliy that $g_0$ is almost maximally symmetric and 
\begin{equation}\label{sh}\Isom_0(S,g_0)=S\rtimes H.\end{equation}

In anticipation of what will follow below, we consider a second metric $g$ on $S$ defined by declaring the basis
$$\B'_\s:=\{A, E_1, \dots, E_4, 2 E_5, E_6\}$$ to be orthonormal.
Observe that $g$ is of the form $\tau^*g_0$ with $\tau\in \Aut(S)$ and thus is also almost maximally symmetric.    Its isometry group is conjugate to that of $g_0$ via $\tau$ and is given by 
\begin{equation}\label{sh'}\Isom_0(S,g_0)=S\rtimes H'\end{equation}
where $H'$ is defined analogously to $H$ but with its restriction to $\R^6$ given by 
$$\begin{bmatrix}\SO(2)&&\\&\SO(2)&\\&&\SO(2)'\end{bmatrix}$$
where $\SO(2)'$ consists of all matrices of the form 
$$\begin{bmatrix}\cos(t)&-2\sin(t)\\\frac{1}{2}\sin(t)&\cos(t)\end{bmatrix}.$$

Define a modification $\rf=(\Id +\varphi)(\s)$ of $\s$ with modification map $\varphi$ satisfying $\varphi(A) =\varphi(E_1)=\dots =\varphi(E_4)=0$, while
$\varphi(E_5)=E_1\wedge E_2$ and $\varphi(E_6)=E_3\wedge E_4$.  (Here $E_1\wedge E_2$ denotes the linear transformation acting on span$(E_1,E_2)$ as $\left[\begin{smallmatrix}0&-1\\1&0\end{smallmatrix}\right]$ and similarly for $E_3\wedge E_4$.)   Observe that $\varphi(\s) < \h\cap \h'=\Der_{\sk}(S,g_0)\cap \Der_{\sk}(S,g).$    Thus the simply-connected Lie group $R$ with Lie algebra $\rf$ acts simply transitively by isometries on both $(S,g_0)$ and $(S,g)$.   We denote by $g_0'$ and $g'$ the induced left-invariant metrics on $R$.   By Theorem~\ref{thm.R admits max sym}, both $g_0'$ and $g'$ are infinitesimally maximally symmetric.   To show that $g_0$ is not almost maximally symmetric, it is enough to show the following:

\smallskip
\noindent \emph{Claim}.  $\Isom_0(R,g) \neq\sigma\Isom_0(R,g_0)\sigma^{-1}$ for any $\sigma\in \Aut(R)$.   

We prove the claim.  Let $B_1$ and $B_2$ be the modifications $B_1:=(\Id+\varphi)(E_5)$ and $B_2:=(\Id+\varphi)(E_6)$ of $E_5$ and $E_6$. We have $\rf=\af'\ltimes \R^4$, where $\af'$ is a 3-dimensional abelian subalgebra with basis $\{A, B_1,B_2\}$ and $\R^4$ has basis $\{E_1,\dots, E_4\}$.  The restrictions of $\ad(B_1)$ and $\ad(B_2)$ to $\R^4$ are given by $E_1\wedge E_2$ and $E_3\wedge E_4$, respectively.   Let $\sigma\in \Aut(R)$.   Then $\sigma_*$ normalizes the nilradical $\R^4$ of $\rf$.   Since $\ad(\sigma_*(A))$ must be semisimple with eigenvalues $0, \pm 1$, we must have $\sigma_*(A)=\pm A$ mod $\R^4$.   It then follows that $\sigma_*$ must either preserve or interchange the $\pm 1$ eigenspaces span$(E_1, E_2)$ and span$(E_3, E_4)$ of $\ad(A)$.   Since each of $\ad(\sigma_*(B_1))$ and $\ad(\sigma_*(B_2))$ must have eigenvalues $0,\pm i$, we must have 
\begin{equation}\label{sigB}\sigma_*(B_1), \sigma_*(B_2) \in \{\pm B_1, \pm B_2\} \mbox{\,\,mod\,\,}\R^4.\end{equation}

The metrics $g_0'$ and $g'$ correspond to the inner products on $\rf$ with orthonormal bases 
$\{A, B_1, B_1, E_1, \dots E_4\}$ and $\{A, 2B_1, B_2, E_1,\dots, E_4\}$.   Since $\Isom_0(R,g_0')=\Isom_0(S,g_0)$ and $\Isom_0(R,g')=\Isom_0(S,g)$, it is now straightforward to verify the claim using Equation~\eqref{sigB}, the fact that $B_1$ and $B_2$ correspond to $E_5$ and $E_6$ under the modification, and the computation of the isometry groups of $\Isom_0(S<g_0)$ and $\Isom_0(S,g)$ given in Equations~\eqref{sh} and \eqref{sh'}.

Concerning Corollary~\ref{cor.uniqueness}, observe that $(S,g_0)$ is a solvsoliton.   Since $(S,g_0)$, $(S,g)$, $(R,g'_0)$ and $(R,g')$ are all isometric, they are all Ricci solitons.    The fact that $g'$ and $g$ are not related by pullback by an automorphism shows that the corollary cannot be strengthened.

\end{example}

\section{Ricci solitons on solvable Lie groups}\label{subsec: background on solvsolitons}
A Riemannian metric is a Ricci soliton if it satisfies the equation $$ric_g = cg + L_X g$$ for some $c\in\mathbb R$ and vector field $X$. Under the Ricci flow, a Ricci soliton metric changes only by diffeomorphisms and homotheties.   We say a  Ricci soliton is   \emph{expanding}, \emph{steady}, or \emph{shrinking} if $c$ is negative, zero, or positive, respectively.   Homogeneous steady Ricci solitons are necessarily flat (i.e., sectional curvature zero); homogeneous, shrinking solitons are necessarily the product of  compact, homogeneous Einstein spaces (with positive Ricci curvature) together with   Euclidean factors \cite{Petersen-Wylie:OnGradientRicciSolitonsWithSymmetry}.

On a Lie group $S$, one may consider the special scenario that a left-invariant metric evolves under Ricci flow by homotheties and automorphisms, not simply diffeomorphisms.  In this case, the Ricci soliton condition becomes
	$$\Ric  = c Id + \frac 1 2 (D+D^t)$$
for some $D\in \mathrm{Der} (\mathfrak s)$.  Here we have presented the equation using the Ricci operator $\Ric_g$ (i.e., the Ricci tensor is given by $ric_g(U,V) = g(\Ric_g(U),v))$.  When the derivation $D$ is symmetric, the above simplifies to $\Ric = c Id +D$ and the metric is called an \emph{algebraic} Ricci soliton.

\begin{defin}~ 
\begin{enumerate} 
\item  A Riemannian solvmanifold for which the Riemannian metric is a Ricci soliton will be called a \emph{Ricci soliton solvmanifold}.   Such solitons are always either flat or expanding.

\item A Riemannian solvmanifold for which the metric is an algebraic Ricci soliton is called a \emph{solvsoliton}.   

\end{enumerate}

\end{defin}

\begin{thm}[B\"ohm--Lafuente]\label{thm: BL}  Every homogeneous expanding Ricci soliton is isometric to a Ricci soliton solvmanifold.
\end{thm}

This theorem is a consequence of the resolution of the Alekseevsky Conjecture \cite{BL2021} and the works \cite{LauretLafuente:StructureOfHomogeneousRicciSolitonsAndTheAlekseevskiiConjecture,
Jablo:HomogeneousRicciSolitonsAreAlgebraic}
which deal with the structure of homogeneous Ricci soliton metrics.

\begin{prop}\label{solvsol uniq}   Not every Ricci soliton solvmanifold is a solvsoliton.  However:
	\begin{enumerate}
	\item \cite{Jablo:HomogeneousRicciSolitonsAreAlgebraic} Every homogeneous Ricci soliton on a solvable Lie group is isometric to an algebraic one,  on a possibly different solvable Lie group; in particular every Ricci soliton solvmanifold is isometric to a solvsoliton.   
	\item \cite{Jablo:HomogeneousRicciSolitonsAreAlgebraic,Jablo:HomogeneousRicciSolitons} 
	On a completely solvable group, a Ricci soliton is necessarily a solvsoliton.   Moreover, every Ricci soliton solvmanifold is isometric to a completely solvable solvsoliton.
\end{enumerate}
\end{prop}

\subsection{Existence and uniqueness results}

\begin{prop}\label{solvsol uniq'}\cite{Lauret:SolSolitons} Up to automorphism and scaling, a simply-connected solvable Lie group can admit at most one solvsoliton metric.

\end{prop}

We will need the following result that does not appear explicitly in the literature.

\begin{lemma}\label{prop: compl solv soliton invariant under compact aut} If $(S,g)$ is a completely solvable solvsoliton, then $\Aut_\orth(S,g)$ is a maximal compact subgroup of $\Aut(S)$.
\end{lemma}

The lemma follows from the proof of \cite[Theorem 4.1]{Jablo:ConceringExistenceOfEinstein}.  While that theorem assumes the Lie group $S$ is unimodular, the relevant part of the proof is valid for any completely solvable Lie group.

The following theorem interprets results of the second author in the language of Definition~\ref{def.scal} and Theorem~\ref{thm: equiv classes}.  

\begin{thm}\label{thm: solitons and equiv classes}  Let $\R\in \Rcal_n$.   If $R$ admits an expanding left-invariant Ricci soliton metric $g$, then every solvable Lie group in the equivalence class of $R$ in $\Rcal_n$ admits a left-invariant Ricci soliton isometric to $g$.   In particular, $(R,g)$ is isometric to a solvsoliton on the unique (up to isomorphism) completely solvable Lie group $S$ in the equivalence class of $R$.   

\end{thm}

\begin{proof} 
By Proposition~\ref{solvsol uniq}, $(R,g)$ is isometric to a completely solvable soliton $(S,g_0)$.  By Theorem~\ref{thm: equiv classes}, $S$ must lie in the equivalence class of $R$.   Lemma ~\ref{prop: compl solv soliton invariant under compact aut} and the argument in the proof of the implication (a) $\implies$ (b) in Theorem~\ref{thm: equiv classes} together show that every $R'$ in the equivalence class admits a left-invariant metric isometric to $g_0$ and thus to $g$.   
\end{proof}

\subsection{Structure of solvsolitons}
Solvsolitons have been extensively studied, see e.g.\ \cite{Will:TheSpaceOfSolsolitonsInLowDimensions,
LauretLafuente:StructureOfHomogeneousRicciSolitonsAndTheAlekseevskiiConjecture,
JP:TowardsTheAlekseevskiiConjecture}.   In this subsection, we review aspects of the structure theory of completely solvable solvsolitons that we will need in the proof of the main theorem.   

We begin with Einstein solvmanifolds.  In 1998, J. Heber made a deep study of what were then called ``standard'' Einstein solvmanifolds $(S,g)$ ; the word ``standard'' means that $(S,g)$ satisfies item (i) in the proposition below.  Later, J. Lauret showed that every Einstein solvmanifold is standard.    The following proposition combines part of Heber's structure theory of standard Einstein solvmanifolds with Lauret's result.   

\begin{prop}\label{prop.heb} 
Let $(S,g)$ be a completely solvable Einstein solvmanifold of negative Ricci curvature and let $\n$ be the nilradical of the Lie algebra $\s$.   Then:
\begin{enumerate}
\item \cite{LauretStandard,LauretNilsoliton}  The orthogonal complement $\af$ of $\n$ relative to the Riemannian inner product is abelian and $\s=\af\ltimes \n$
\item \cite{Heber} $\ad_\s(\af)$ consists of semisimple derivations.

\item \cite{Heber} Let $H$ be the unique element of $\af$ such that $g(H, A)=\mytrace(\ad(A))$ for all $A\in\af$ .  Then there exists $\lambda\in\R^+$ such that the real parts of the eigenvalues of $\lambda\, \ad(H)^\R|_\n$ are positive integers.   (The derivation $\ad(H)|_{\n}$ is sometimes called the \emph{Einstein derivation}. 
\end{enumerate}
\end{prop}

\begin{defin}\label{def.pre-Einst} (Nikolayevski \cite{Nikolayevsky:EinsteinSolvmanifoldsandPreEinsteinDerivation}.) 
 A derivation $\varphi \in Der(\mathfrak g)$ of a Lie algebra $\g$ is said to be a \textit{pre-Einstein derviation} if it is semisimple as an element of $End(\mathfrak g)$ with all eigenvalues real, and satisfies
\begin{equation}\label{eqn:pre-Einstein deriv} \mytrace(\varphi A) = \mytrace(A)  \quad \mbox{ for all } A\in \Der(\g)
\end{equation}

 \end{defin}

 \begin{prop} \cite[Theorem 1]{Nikolayevsky:EinsteinSolvmanifoldsandPreEinsteinDerivation}\label{prop.nik}
 $\text{}$
\begin{enumerate}
 \item Every Lie algebra admits a pre-Einstein derivation $\varphi$, unique up to automorphism.   The eigenvalues of $\varphi$ are rational.
  \item The pre-Einstein derivation lies in the center of a maximal reductive subalgebra of $\Der(\n)$ and is invariant under a maximal compact subgroup of $Aut(\mathfrak n)$.   
 \item If $(S,g)$ is a completely solvable Einstein solvanifold of negative Ricci curvature, then the associated Einstein derivation (see Proposition~\ref{prop.heb}(iii)) is a positive multiple of a pre-Einstein derivation of $\nilrad(\s)$.

\end{enumerate}
 
 \end{prop}
 The second item does not appear in the statement of \cite[Theorem 1]{Nikolayevsky:EinsteinSolvmanifoldsandPreEinsteinDerivation} but rather in its proof.
 
\begin{remark}\label{rem.nik} By Proposition~\ref{prop.nik} and the conjugacy of maximal reductive   subalgebras of $\Der(\n)$, one can always choose a pre-Einstein derivation that commutes with any given reductive subalgebra of $\Der(\n)$.

\end{remark}

We now turn to completely solvable Ricci solitons:

\begin{prop}\cite{Lauret:SolSolitons}\label{proplauret}
Let $(S,g)$ be a completely solvable expanding solvsoliton, let $N$ be the nilradical of $S$, and let $\s$ and $\n$ be the Lie algebras of $S$ and $N$.  Then
\begin{enumerate}
\item\label{cond1}   $\s$ can be written as a semi-direct sum $\mathfrak s = \mathfrak a \ltimes \mathfrak n$ with $\af$ an abelian subalgebra; 
\item  $\mathfrak n$ admits a Ricci soliton metric (called a nilsoliton in the literature);
\item $\af$ is $\ad_\s$-reductive.
\item If $(S,g)$ is not itself Einstein, let $\delta$ be a pre-Einstein derivation of $\n$ that commutes with $\ad_\s(\af)_{|\n}$.   Let $\s'$ be the one-dimensional extension of $\s$ given by $\s'=\s\ltimes \R\delta$, and let $S'$ be the simply-connected solvable extension of $S$ with Lie algebra $\s'$.   Then the Ricci soliton $g$ extends to an Einstein metric on $S'$ of negative Ricci curvature.

\end{enumerate}

\end{prop}

\section{Proof of Theorem~\ref{main} }

By Theorems~\ref{thm.R admits max sym} and \ref{thm: solitons and equiv classes}, along with Corollary~\ref{cor.uniqueness}, it suffices to prove the following special case:

\begin{thm}\label{special case} Solvsolitons on simply-connected completely solvable Lie groups have maximal infinitesimal symmetry.
\end{thm}

Since solvsolitons on $S$ are unique up to automorphism, all solvsolitons on $S$ have isomorphic isometry algebras.     Thus Theorem~\ref{special case} is equivalent to the following:

\begin{thm}\label{special case V2}  Let $S$ be a simply-connected completely solvable Lie group that admits a solvsoliton, let $g$ be an arbitrary left-invariant metric on $S$, and write $\g:=\cisom(S,g)$.  Then there exists a connected Lie group $G$ with Lie algebra $\g$ and a solvsoliton $g_0$ on $S$ such that $G<\Isom_0(S,g_0)$.   

\end{thm}

The proof  consists of three steps:

\begin{itemize}
\item[1.] Constructing the candidate Lie group $G$.
\item[2.] Showing that the pre-Einstein derivation of the nilradical $\n$ of $\s$ extends to a derivation of $\g$.
\item[3.] By Proposition~\ref{proplauret}, unless $S$ itself admits an Einstein metric, the pre-Einstein derivation yields a one-dimensional extension $S'$ of $S$ that admits an Einstein metric $g_E$ whose restriction $g_0$ to $S$ is a Ricci soliton.   In the latter case, we will apply step 2 and the maximal symmetry of Einstein metrics (see\cite{GordonJablonski:EinsteinSolvmanifoldsHaveMaximalSymmetry}) to obtain a one-dimensional extension $G'$ of $G$ that acts isometrically on $g_E$.   This will enable us to conclude that $G$ acts isometrically on $(S,g_0)$.  
\end{itemize}

\subsection{Step 1: Construction of the candidate Lie group $G$}\label{sec: step 2 - linearizing the isom group}~

\begin{remark} Since the second author \cite[Proposition 7.1]{Jablo:StronglySolvable} showed that the isometry group of any solvsoliton is a linear Lie group, we will construct a linear Lie group as our candidate for $G$.   While the result in \cite{Jablo:StronglySolvable} motivates our proof, we do not use this result directly.   In particular, Theorem~\ref{special case V2} and its proof yield a new proof that the 
isometry group of any solvsoliton is linear.  
\end{remark}

 Let $G^*=\Isom_0(S,g)$.    By Propositions~\ref{stdpos} and \ref{lem.s1}, we have a Levi decomposition $G^*=G_1^*G_2^*$, an Iwasawa decomposition $G_1^*=K^*S_1$ and a decomposition $G_2^*=T\ltimes S_2$, where $T$ is a torus that commutes with $G_1$ and is contained in the isotropy group $L^*$ of $g$ at the identity.   In particular, $T$ acts effectively on $S_2$, preserving the metric $g$.   Moreover, $S\simeq S_1\ltimes S_2$, and $S_2$ is a normal subgroup of $G$. 

 Set $G_2 :=G_2^*$.  Let  $\widetilde{G}_1$ be the simply-connected cover of $G_1^*$.  We can lift the adjoint action of $G_1^*$ on $S_2$ to an action of  $\widetilde{G}_1$ commuting with $T$ and thus form the semi-direct product $\widetilde{G}_1\ltimes G_2$.   Writing $\widetilde{G}_1=\widetilde{K}S_1$, the center of $\widetilde{G}_1\ltimes G_2$ intersects $\widetilde{K}T$ in a discrete subgroup $D$ with compact quotient. 
 Set 
 $$G= (\widetilde{G}_1T/D)\ltimes S_2.$$
Since $T$ acts effectively on $S_2$, we have $T\cap D = \{e\}$ and so $G_2$ can be realized as a subgroup of $G$.  We note that $G$ has Levi decomposition $G=G_1G_2$ with $G_1= \widetilde{G}_1/(D\cap \widetilde{K})$, but it is possible that $G_1\cap T$ is non-trivial.  Further,  $G_1$ has Iwasawa decomposition $G_1=KS_1$ with
 $K= \widetilde{K}/(D\cap  \widetilde{K})$.  Set
 $$L= (\widetilde{K}\times T)/D.$$
 Then $L$ is compact and $G=LS$ with $L\cap S=\{e\}$.  We note that, by design, ${Z(G)\cap L=\{e\}}$.  
 
We show $G$ acts effectively on $G/L$, equivalently that $L$ contains no non-trivial normal subgroups of $G$.  Both $\lf$ and $\lf^*$ (the Lie algebras of $L$ and $L^*$) are compactly embedded subalgebras of $\g$ complementary to $\s$.  Since $\s$ is completely solvable, $\uf:=\lf\oplus Z(\g)$ and $\uf^*=\lf^* \oplus Z(\g)$ are maximal compactly embedded subalgebras of $\g$ and thus are mutually conjugate.   Since $G^*$ acts effectively on $S$, the isotropy algebra $\lf^*$ contains no proper $\g$-ideals, so $Z(\g)$ is a maximal $\g$-ideal in $ \uf^*$ and consequently also in $\uf$.  Hence $\lf$ contains no $\g$-ideals.  Thus any subgroup of $L$ that is normal in $G$ must be discrete, hence central in $G$.  Since $Z(G)\cap L=\{e\}$, it follows that the action is effective.

\subsection{Step 2: Extending the pre-Einstein derivation to $\g$}

We will state the extension result in a somewhat more general context than needed.   It would be interesting to know whether any of the hypotheses of the proposition can be removed.

     \begin{prop}\label{prop:  form_of_pre-Einstein} Assume:
     \begin{enumerate} \item $S$ is a simply-connected completely solvable Lie group of the form $S=A\ltimes N$ where $A$ is abelian, $N$ is the nilradical, and $\ad_\s(\af)$ is reductive.  
     \item Some (hence every) pre-Einstein derivation of $\n$ is non-singular. 
     \item $G$ is a connected Lie group of the form $G=LS$ where $L$ is compact, $L\cap S$ is trivial and $L$ contains no normal subgroups of $G$.  
     \end{enumerate}
   Then there exists a pre-Einstein derivation of $\n$ that extends to a derivation of $\g$.    Moreover, letting $A=A_1A_2$, $N=N_1N_2$ and $G=G_1G_2=(KA_1N_1)(T\ltimes A_2N_2)$ be the decompositions guaranteed by Propositions~\ref{stdpos}, \ref{lem.s1} and \ref{lem.s2}, the pre-Einstein derivation leaves each of $\n_1$ and $\n_2$ invariant and is of the form
	$$\varphi =\ad_\n(X)+\begin{bmatrix}0&0\\0&\delta\end{bmatrix}$$
for some $X\in\af_1$ and some pre-Einstein derivation $\delta$ of $\n_2$ that commutes with the adjoint action of $\g_1+\tf+\af_2$ on $\n_2$.
\end{prop} 
   
Note that Propositions~\ref{prop.heb}, \ref{prop.nik}, and \ref{proplauret} guaranteed that every completely solvable Lie group that admits a solvsoliton satisfies the hypotheses of Proposition~\ref{prop:  form_of_pre-Einstein}.

\begin{lemma}\label{lemma:  derivations of n1 commuting with a1+m} Let $\m$ denote the centralizer of $\af_1$ in $\kf$, and let $\m_1=[\m,\m]$.   Then the only derivations of $\n_1$ that commute with $(\ad_{\g_1}(\af_1+\m))_{|\n_1}$ are those in $(\ad_{\g_1}(\af_1+Z(\m_1)))_{|\n_1}$.
\end{lemma}

\begin{proof}
We change notation (in this proof only) in order to be consistent with Helgason \cite{Helgason}.  We let $\g_0$ be a semisimple Lie algebra of noncompact type, $\g_0=\kf_0+\af_0+\n_0$ an Iwasawa decomposition and $\m_0$ the centralizer of $\af_0$ in $\kf_0$.   Let $D$ be a derivation of $\n_0$ that commutes with $\af_0+\m_0$.   Let $\g=\g_0^\C$ denote the complexification of $\g_0$, and write $\n=\n_0^\C$, etc.  so $\g=\kf+\af+\n$.  Extend $D$ to a derivation of $\n$, still denoted $D$, by making it complex linear.   

We show that $D$ extends to a derivation of $\g$.   Recall (see \cite[page 260]{Helgason}) how the Iwasawa decomposition is constructed.   One has a Cartan decomposition $\g_0=\kf_0+\p_0$, a Cartan subalgebra $\h$ of $\g$ satisfying $\h\cap \g_0=\h_{\p_0} +\h_{\kf_0}$ where $\h_{\p_0}$ is a maximal abelian subalgebra in $\p_0$ (this will be our $\af_0$) and $\h_{\kf_0}<\kf_0$.   Moreover $\h_{\R}:= i\h_{\kf_0} + \h_{\p_0}$ is a real form of $\h$.  Take compatible orderings on $\h_{\p_0}$ and  
$\h_{\R}$  
and let $\Delta< \h_{\R}^*$ be the associated system of positive roots.   Following Helgason's notation, write $\Delta = P_+ \cup P_-$ where $P_-$ consists precisely of those positive roots that vanish on $\h_{\p_0}$.  For $\alpha\in\Delta$, let $\g^\alpha$ be the corresponding root space in $\g$ (i.e., the complex root space).  
Then 
$\g=\kf+\af+\n$ where $\af:=\h_{\p_0}^\C$, and 
$$\n=\sum_{\alpha\in P_+}\g^\alpha.$$
The subspace $\m=\m_0^\C$ of $\kf$ is given by 
$$\mf=\h_{\kf_0}^\C + \sum_{\alpha\in P_-}(\g^\alpha +\g^{-\alpha}).$$
Note that 
$$\g=\m+\af +\n + \sum_{\alpha\in P_+}\g^{-\alpha}.$$
  
  Since $D$ commutes with $\ad(\af +\m)$ and thus with $\ad(\h)$, it preserves each root space $g^\alpha$, $\alpha \in P_+$.  Since the root spaces have complex dimension 1, we have $D|_{\g_\alpha}=c_\alpha\Id$ for each $\alpha\in P_+$ where the $c_\alpha$'s are constants.  In particular, $D$ is a semisimple derivation.  Moreover, the fact that $D$ commutes with $\ad(\m)$ implies that if $\gamma\in P_-$, $\alpha\in P_+$ and if $\alpha +k\gamma \in P_+$ with $k\in\{\pm 1\}$, then $c_{\alpha+k\gamma}=c_\alpha$.   The fact that $D$ is a derivation on $\mathfrak n$ is equivalent to the condition that $c_{\alpha+\beta}=c_\alpha +c_\beta$ whenever $\alpha, \,\beta$ and $\alpha +\beta$ all belong to $P_+$.  Extend $D$ to $\g$ by linearity and by setting $D|_{\m +\af}=0$ and $D|_{\g^{-\alpha}}=-c_{\alpha}\Id$ for $\alpha\in P_+$.  Computing the various possible brackets among the positive and negative weight spaces of $\mathfrak g$, one sees that $D$ is a derivation of $\g$.
  
  Since all derivations of the semisimple Lie algebra $\g$ are inner, we have $D=\ad(Y)$ for some $Y\in \g$.  Since $D$ normalizes $\n$, $D$ is semisimple and $D$ commutes with $\af+\m$, we must have $Y\in \af+Z(\m)$.

Finally, since $D$ preserves $\mathfrak n_0$ and $\mathfrak a+\mathfrak m$ is the complexification of $\mathfrak a_0+\mathfrak m_0$, we see that we must have $Y\in \mathfrak a_0+\mathfrak m_0$ and so $Y\in \mathfrak a_0+Z(\mathfrak m_0)$, as desired.

\end{proof}

We gather here several results from \cite{GordonJablonski:EinsteinSolvmanifoldsHaveMaximalSymmetry} that will be needed below.  In preparation for the third result, let $\Der(\n_2) = \Der_1(\n_2) + \Der_2(\n_2)$ be a Levi decomposition.   Note that since $\Der(\n_2)$ is the Lie algebra of an algebraic group, the radical $\Der_2(\n_2)$ splits as a semi-direct sum 
	$$\Der_2(\n_2)=\Der^{ab}(\n_2)+\nilrad(\Der(\n_2))$$
where $\Der^{ab}$ is an abelian subalgebra commuting with $\Der_1(\n_2)$.   The subalgebra $\Der_1(\n_2) + \Der^{ab}(\n_2)$ is a maximal fully reducible subalgebra of $\Der(\n_2)$ (i.e., maximal among all subalgebras of $\Der(\n_2)$ that act fully reducibly on $\n_2$), and  $\Der^{ab}(\n_2)$  consists of semisimple derivations of $\n_2$.   

\begin{lemma}\label{gjlemma}~ 
\begin{enumerate}
\item \cite[Lemma 5.4]{GordonJablonski:EinsteinSolvmanifoldsHaveMaximalSymmetry}
 Let $H$ be a semisimple Lie group of noncompact type and let $\rho: H\to GL(V)$ be a finite-dimensional representation.  Let $H=KAN$ be an Iwasawa decomposition.  If an element $T\in End(V)$ commutes with all elements of $\rho(AN)$, then it commutes with all of $\rho(H)$.  
 \item \cite[Lemma 5.5]{GordonJablonski:EinsteinSolvmanifoldsHaveMaximalSymmetry} Let $H$ be a connected Lie group and $N$ its nilradical.   Let $W\in\h$, and suppose that $\ad(W)|_\n:\n\to\n$ is a non-singular derivation.   Then the orbit of $W$ under $\Ad_H(N)$ is given by
$$\Ad_H(N)(W)=\{W+X: X\in\n\}.$$
\item \cite[Lemma 5.7]{GordonJablonski:EinsteinSolvmanifoldsHaveMaximalSymmetry} Let $\sigma$ be a semisimple derivation of $\n_2$ with real eigenvalues.  Let $\mathfrak r$ be any fixed choice of subalgebra of $\Der_1(\n_2) + \Der^{ab}(\n_2)$ that contains $\sigma$ and  $\Der^{ab}(\n_2)$.  If $ \mytrace(\sigma D)=\mytrace(D) $ for all $D \in \mathfrak r$, then $\sigma$ is a pre-Einstein derivation of $\n_2$.
\end{enumerate}
\end{lemma}

\begin{proof}[Proof of Proposition~\ref{prop:  form_of_pre-Einstein}]
We let $\rho$ denote the action of $\g_1+\tf +\af_2$ on $\n_2$.

We make a series of claims, proved along the way, which build up to the proof of the proposition.

\begin{enumerate} 
\item Let $D$ be any derivation of $\n$ such that $D(\n)\subset \n_2$.    For $X_1\in \n_1$, we have \begin{equation}\label{brack}[D_{|\n_2},\rho(X_1)]=\ad_{\n_2}(D(X_1)).\end{equation}
Thus the restriction of $D$ to $\n_2$ commutes with $\rho(\n_1)$ if and only if $D(\n_1)\subset Z(\n_2)$.

\item Since $\af+\m+\tf$ acts via semisimple derivations of $\n$, we may choose a pre-Einstein derivation $\psi$ that commutes with $\ad_\n(\af+\m+\tf)$.  (See Remark~\ref{rem.nik}.)
$\psi$ is of the form 

$$\psi=\ad_\n(X)+ D$$
for some $X\in \af_1$
and some derivation $D$ of the form
$$D=\begin{bmatrix}0&0\\ \beta&\delta
\end{bmatrix}$$
relative to the decomposition $\n=\n_1+\n_2$.
Moreover $\delta$ commutes with $\rho(\af + \m +\mathfrak t)$.  

(We are \emph{not} claiming at this point that $\delta$ is a pre-Einstein derivation of $\n_2$.)
\begin{proof} We first claim that $\psi$ preserves $\n_2$ and thus has the ``lower triangular'' block form 
$$\psi=\begin{bmatrix}\gamma&0\\\beta&\delta \end{bmatrix}$$

Choose a base $\Delta$ for the root system of $\g_1$ and let $\Delta^+$ be the positive roots.  For $\alpha\in\Delta$, let $H_\alpha$ satisfy $\alpha(H)=\langle H_\alpha, H\rangle$ for all $H\in \af_1$.   Let $\n_2^-$ be the subspace of $\n_2$ spanned by all weight spaces $V_\beta$ of $\rho(\af_1)$ for which $\beta(H_\alpha)<0$ for at least one $\alpha\in\Delta^+$ and let $\n_2^0$ be the 0-weight space of $\rho(\af_1)$.   Then $\n_2^-+\n_2^0$ generaties $\n_2$ as a $\rho(\g_1)$-module.  (Indeed let $V_\beta$ be a weight space not in $\n_2^-+\n_2^0$.   Then there exists $\alpha\in \Delta^+$ with $\beta(H_\alpha)>0$.  Consider the subalgebra $\s_\alpha\simeq \slf(2,\R)$ spanned by $H_\alpha, X_\alpha, Y_\alpha$, with $X_\alpha$ in the $\alpha$ root space in $\n_1$ and $Y_\alpha$ in the $-\alpha$ root space.   Then the $\rho(\s_\alpha)$-module in $\n_2$ spanned by $V_\beta$ is generated under the action of $\rho(X_\alpha)$ by negative eigenspaces of $\alpha(H_\alpha)$.)  

Now $\n_2^-+\n_2^0$ is invariant under $\psi$ since $\psi$ commutes with $\rho(\af_1)$.   Since $\n_2$ is an ideal in $\n$ and $\psi$ is a derivation of $\n$, the fact that $\n_2^-+\n_2^0$ generates $\n_2$ under the action of $\n_1$ implies that $\n_2$ is invariant under $\psi$.

Finally, Lemma \ref{lemma:  derivations of n1 commuting with a1+m} and the fact that $\psi$ has only real eigenvalues implies that $\gamma=\ad_{\n_1}(X)$ for some $X\in\af_1$.  Defining $D=\psi-\ad_\n(X)$, then $D$ has the desired form.
\end{proof}

\item There exists $U\in\n_2$  such that $\delta+\ad_{\n_2}(U)$ commutes with $\rho(\g_1+\af_2+\tf)$.   Furthermore, we may choose $U\in\n_2^0$, the zero weight space of $\rho(\mathfrak a_1)$ in $\mathfrak n_2$.

\begin{proof} By Equation~(\ref{brack}), applied to the derivation $D$ in (ii), and the fact that $\delta$ commutes with $\rho(\af +\mf+\tf)$, we have that $\R\delta+\ad(\n_2)$ is a subalgebra of $\Der(\n_2)$ normalized by $\rho(\af_1+\n_1 +\m+ \tf +\af_2)$.  Consider the projection $\pi: \Der(\n_2)\to \Der(\n_2)/\ad(\n_2)$.  Again using Equation~(\ref{brack})), we see that $\pi\circ\delta$ commutes with $\pi\circ \rho(\af_1+\n_1+\m+ \tf+\af_2)$.   Lemma~\ref{gjlemma}(i) along with Remark~\ref{rem: unique comp solv} then imply that $\pi\circ\delta$ commutes with $\pi\circ\rho(\g_1+\tf +\af_2)$.  Thus $[\rho(\g_1+\tf+\af_2),\delta]<\ad(\n_2)$.   The first statement follows from complete reducibility of the action of $\rho(\g_1+\tf+\af_2)$.

Next, the fact that $\delta$ and $\delta+\ad_{\n_2}(U)$ both commute with $\rho(\mathfrak a_1)$ implies that $\ad_{\n_2}(U)$ commutes with $\rho(\mathfrak a_1)$.  Thus $\ad_{\mathfrak n_2}  [Y,U] = 0$ for all $Y\in\mathfrak a_1$, i.e.~$[Y,U]\in Z(\mathfrak n_2)$ and so $\mathbb R(U)+Z(\mathfrak n_2)$ is a $\rho(\mathfrak a_1)$ invariant subspace.  Now, as $\rho(\mathfrak a_1)$ acts fully reducibly, we may replace $U$ by $U+Z$ for a suitably chosen $Z\in \n_2$ so that the second statement is satisfied.
\end{proof}

\item  \label{label for remark below}  We may conjugate $\psi$ by an element $\Ad(x)$ of $\Ad(N)$ with $x\in N_2$ to obtain a new pre-Einstein derivation $\tau$ of the form 
$$\tau=\ad_\n(X)+D'$$
where $D'$ is of the form
$$D'=\begin{bmatrix}0&0\\ \beta'&\delta'
\end{bmatrix}$$
with $\delta'=\delta+ \ad(U)$.   In particular, $\delta'$ commutes with $\rho(\g_1+\tf+\af_2)$.  Thus by (i), $\beta'(\n_1)<Z(\n_2)$.

\begin{proof} Since $\psi$ commutes with $\rho(\mathfrak a_1)$, it preserves the zero weight space $n_2^0$.  Since $\n$ admits a nilsoliton, $\psi$ is non-singular (see Propositions~\ref{prop.heb}, \ref{prop.nik} and \ref{proplauret}) and thus so is its restriction $\rho(X)+\delta$ to $\n_2^0$.  

Define $\h=\R(\rho(X)+\delta)+\ad(\n_2^0)$; note, this is a subalgebra.    By Lemma~\ref{gjlemma}(ii), $\rho(X)+\delta$ is conjugate to $\rho(X)+\delta +\ad(U)$ by $\Ad_{N_2}(x)$ for some $x\in \exp(\mathfrak n_2^0)\subset N_2$.   Now conjugate $\psi$ by $\Ad_N(x)$.   
\end{proof}

{\bf Aside.}   When we conjugated $\psi$ by $\Ad(x)$, we may have lost the property that $\tau$ commutes with $\ad(\af_1)$.   The only part of $\tau$ that may not commute with $\ad(\af_1)$, however, must be $\begin{bmatrix}0&0\\\beta' & 0\end{bmatrix}$.

\item $\delta'$ is a pre-Einstein derivation of $\n_2$.   
\begin{proof}
The fact that $\tau$ is semisimple implies that its restriction $\rho(X)+\delta'$ to $\n_2$ is also semisimple.  Then so is $\delta'$ since $\rho(X)$ is semisimple and commutes with $\delta'$.  We need to show that $\mytrace(\delta'D)=\mytrace(D)$ for all $D\in \Der(\n_2)$.   By Lemma~\ref{gjlemma}(iii), it's enough to consider the case that $D$ commutes with $\rho(\g_1)$.   In this case, $D$ extends trivially to $\n$, vanishing on $\n_1$.  Continue to denote the extension by $D$.  Since $\tau$ is a pre-Einstein derivation on $\n$, we have that $\mytrace(\tau D)=\mytrace(D)$.   But $\mytrace(\tau D)=\mytrace((\rho(X)+\delta')D).$    Now $D$ commutes with $\rho(\g_1)$, so $\rho(\g_1)$ leaves each eigenspace of $D$ invariant.   Since $\g_1$ is semisimple, the trace of $\rho(X)$ on each such eigenspace is zero, and hence $\mytrace(\rho(X)D)=0$.   Thus we have 
$$\mytrace(D)=\mytrace(\tau D)=\mytrace(\delta' D)$$
and statement (v) follows.

\end{proof}

 \item Define $\varphi=\ad_\n(X)+\begin{bmatrix}0&0\\ 0&\delta'\end{bmatrix}$.   Then $\varphi\in\Der(\n)$.  It is the desired pre-Einstein derivation.
 \begin{proof} $\varphi$ is semisimple, so we only need to check the trace condition.  Note that $\varphi$ commutes with $\ad(\af_1)$, which is a subalgebra of $\Der(\n)$ that acts fully reducibly.  By Lemma~\ref{gjlemma}(iii), it is enough to check the trace condition on derivations $D$ of $\n$ that commute with $\ad(\af_1)$.   By the proof of statement (ii), such derivations $\D$ must have lower triangular block form and hence the product $\begin{bmatrix}0&0\\\beta'&0\end{bmatrix}D$ has trace zero.   Hence we have
 $$\mytrace(\varphi D)=\mytrace(\tau D)=\mytrace(D)$$
 where the second equation follows from the fact that $\tau$ is a pre-Einstein derivation.    This completes the proof.
 \end{proof}
 
 \end{enumerate}

This completes the proof of Proposition \ref{prop:  form_of_pre-Einstein}.

\end{proof}

\subsection{Step 3: Employing the Einstein metric}

\begin{proof}[Completion of Proof of Theorem~\ref{main}]
In view of Theorem~\ref{oldmain}, it's enough to consider the case that the (unique up to automorphism) solvsoliton on $S$ is not Einstein.  By Proposition~\ref{prop:  form_of_pre-Einstein}, there exists a pre-Einstein derivation $\delta$ of $\n_2$ that commutes with the adjoint action on $\n_2$ of $\g_1+\tf +\af_2$, and there exists an element $X\in\af_1$ such that $\varphi = \ad_\n(X)+\delta$ is a pre-Einstein derivation of $\n$.  

The derivation $\delta$ extends to a derivation of $\g$ acting trivially on $\g_1+\tf+\af_2$ and thus we may form the semi-direct product $\h:=\g\rtimes\R\delta$.   Let $\s'=\s\rtimes\R\varphi\, =\s\rtimes\R\delta= \s_1+(\s_2\rtimes\R\delta)$.   By Proposition~\ref{proplauret}, $\s'$ is the Lie algebra of a simply-connected solvable Lie group $S'=S_1S_2'$ containing $S$ that admits a left-invariant Einstein metric.  Here $S_2'$ is the simply-connected one-dimensional extension of $S_2$ with Lie algebra $\s_2' = \s_2\rtimes\R\delta$.

Since  $\g_1+\tf$ commutes with $\delta$ we have the semi-direct product structure $\h=(\g_1+\tf )\ltimes \s_2'$; now we can construct a one-dimensional extension $H$ of $G$ given by $H=(G_1\times T)\ltimes S'_2$.  Note that $H=(KT)S'$ and $KT\cap S'=\{e\}$.  By Theorem~\ref{oldmain}, the Lie group $H$ acts isometrically on $S'$ relative to some choice of Einstein metric $g_E$.  The restriction of $g_E$ to $S$ is a solvsoliton $g_0$.  The action of $H$ on $S'$ restricts to an isometric action of $G$ on $S$ relative to $g_0$; i.e., $G<\Isom_0(S,g_0)$.    
\end{proof}

\providecommand{\bysame}{\leavevmode\hbox to3em{\hrulefill}\thinspace}
\providecommand{\MR}{\relax\ifhmode\unskip\space\fi MR }
\providecommand{\MRhref}[2]{%
  \href{http://www.ams.org/mathscinet-getitem?mr=#1}{#2}
}
\providecommand{\href}[2]{#2}


\begin{thebibliography}{Jab15b}

\bibitem[AL17]{Arroyo-Lafuente:TheAlekseevskiiConjectureInLowDimensions}
Romina~M. Arroyo and Ramiro~A. Lafuente, \emph{The {A}lekseevskii conjecture in
  low dimensions}, Math. Ann. \textbf{367} (2017), no.~1-2, 283--309.
  \MR{3606442}

\bibitem[Ale75]{Alekseevski:RiemSpacesOfNegCurv}
D.~V. Alekseevski{\u\i}, \emph{Homogeneous {R}iemannian spaces of negative
  curvature}, Mat. Sb. (N.S.) \textbf{96(138)} (1975), 93--117, 168.

\bibitem[BL18]{Bohm-Lafuente:HomogeneousEinsteinMetricsOnEuclideanSpacesAreEinsteinSolvmanifolds}
Christoph B{\"o}hm and Ramiro Lafuente, \emph{Homogeneous {E}instein metrics on
  {E}uclidean spaces are {E}instein solvmanifolds}, arXiv:1811.12594 (2018).

\bibitem[BL19]{Bohm-Lafuente:TheRicciFlowOnSolvmanifoldsOfRealType}
Christoph B\"{o}hm and Ramiro~A. Lafuente, \emph{The {R}icci flow on
  solvmanifolds of real type}, Adv. Math. \textbf{352} (2019), 516--540.
  \MR{3964154}

\bibitem[BL21]{BL2021}
Christoph B\"{o}hm and Ramiro~A. Lafuente, \emph{Non-compact Einstein manifolds with symmetry}, 
    arXiv: 2107.04210 (2021).
     


\bibitem[EJ23]{EpsteinJablonski:MaximalSymmetryAndSolvmanifolds}
Jonathan Epstein and Michael Jablonski, \emph{Maximal symmetry and
  solvmanifolds}, in progress (2023).

\bibitem[GJ19]{GordonJablonski:EinsteinSolvmanifoldsHaveMaximalSymmetry}
Carolyn Gordon and Michael Jablonski, \emph{Einstein solvmanifolds have maximal
  symmetry}, Journal of Differential Geometry \textbf{111} (2019), no.~1,
  1--38.

\bibitem[Gor81]{Gordon:NilpRad}
Carolyn Gordon, \emph{Transitive {R}iemannian isometry groups with nilpotent
  radicals}, Ann. Inst. Fourier (Grenoble) \textbf{31} (1981), no.~2, vi,
  193--204. \MR{617247}

\bibitem[GW85]{GordonWilson:TheFineStructureOfTransitiveRiemannianIsometryGroups}
Carolyn~S. Gordon and Edward~N. Wilson, \emph{The fine structure of transitive
  {R}iemannian isometry groups. {I}}, Trans. Amer. Math. Soc. \textbf{289}
  (1985), no.~1, 367--380. \MR{779070 (86g:53056)}

\bibitem[GW88]{GordonWilson:IsomGrpsOfRiemSolv}
\bysame, \emph{Isometry groups of {R}iemannian solvmanifolds}, Trans. Amer.
  Math. Soc. \textbf{307} (1988), no.~1, 245--269.

\bibitem[Heb98]{Heber}
Jens Heber, \emph{Noncompact homogeneous {E}instein spaces}, Invent. Math.
  \textbf{133} (1998), no.~2, 279--352.

\bibitem[Hel01]{Helgason}
Sigurdur Helgason, \emph{Differential geometry, {L}ie groups, and symmetric
  spaces}, Graduate Studies in Mathematics, vol.~34, American Mathematical
  Society, 2001.

\bibitem[Jab11]{Jablo:ConceringExistenceOfEinstein}
Michael Jablonski, \emph{Concerning the existence of {E}instein and {R}icci
  soliton metrics on solvable lie groups}, Geometry \& Topology \textbf{15}
  (2011), no.~2, 735--764.

\bibitem[Jab14]{Jablo:HomogeneousRicciSolitonsAreAlgebraic}
\bysame, \emph{Homogeneous {R}icci solitons are algebraic}, Geom. Topol.
  \textbf{18} (2014), no.~4, 2477--2486. \MR{3268781}

\bibitem[Jab15a]{Jablo:HomogeneousRicciSolitons}
\bysame, \emph{Homogeneous {R}icci solitons}, J. Reine Angew. Math.
  \textbf{699} (2015), 159--182. \MR{3305924}

\bibitem[Jab15b]{Jablo:StronglySolvable}
\bysame, \emph{Strongly solvable spaces}, Duke Math. J. \textbf{164} (2015),
  no.~2, 361--402. \MR{3306558}

\bibitem[Jab19]{Jablo:MaximalSymmetryAndUnimodularSolvmanifolds}
\bysame, \emph{Maximal symmetry and unimodular solvmanifolds}, Pacific J. Math.
  \textbf{298} (2019), no.~2, 417--427. \MR{3936023}

\bibitem[Jab21]{Jablo:SurveyHomogeneousEinsteinManifolds}
\bysame, \emph{Survey: Homogeneous {E}instein manifolds}, arXiv:2111.09782
  (2021).

\bibitem[JP17]{JP:TowardsTheAlekseevskiiConjecture}
Michael Jablonski and Peter Petersen, \emph{A step towards the {A}lekseevskii
  conjecture}, Math. Ann. \textbf{368} (2017), no.~1-2, 197--212. \MR{3651571}

\bibitem[Lau01]{LauretNilsoliton}
Jorge Lauret, \emph{Ricci soliton homogeneous nilmanifolds}, Math. Ann.
  \textbf{319} (2001), no.~4, 715--733.

\bibitem[Lau10]{LauretStandard}
\bysame, \emph{Einstein solvmanifolds are standard}, Ann. of Math. (2)
  \textbf{172} (2010), no.~3, 1859--1877.

\bibitem[Lau11]{Lauret:SolSolitons}
\bysame, \emph{Ricci soliton solvmanifolds}, J. Reine Angew. Math. \textbf{650}
  (2011), 1--21.

\bibitem[LL14]{LauretLafuente:StructureOfHomogeneousRicciSolitonsAndTheAlekseevskiiConjecture}
Ramiro Lafuente and Jorge Lauret, \emph{Structure of homogeneous {R}icci
  solitons and the {A}lekseevskii conjecture}, J. Differential Geom.
  \textbf{98} (2014), no.~2, 315--347. \MR{3263520}

\bibitem[Mos56]{Mostow:FullyReducibleSubgrpsOfAlgGrps}
G.~D. Mostow, \emph{Fully reducible subgroups of algebraic groups}, Amer. J.
  Math. \textbf{78} (1956), 200--221.

\bibitem[Nik11]{Nikolayevsky:EinsteinSolvmanifoldsandPreEinsteinDerivation}
Y.~Nikolayevsky, \emph{Einstein solvmanifolds and the pre-{E}instein
  derivation}, Trans. Amer. Math. Soc. \textbf{363} (2011), 3935--3958.

\bibitem[Oze77]{Oze77}
Hideki Ozeki, \emph{On a transitive transformation group of a compact group
  manifold}, Osaka Math. J. \textbf{14} (1977), no.~3, 519--531. \MR{461377}
  
\bibitem[PW09]{Petersen-Wylie:OnGradientRicciSolitonsWithSymmetry}
Peter Petersen and William Wylie, \emph{On gradient {R}icci solitons with
  symmetry}, Proc. Amer. Math. Soc. \textbf{137} (2009), no.~6, 2085--2092.
  \MR{2480290 (2010a:53073)}

\bibitem[Wil11]{Will:TheSpaceOfSolsolitonsInLowDimensions}
Cynthia Will, \emph{The space of solvsolitons in low dimensions}, Ann. Global
  Anal. Geom. \textbf{40} (2011), no.~3, 291--309. \MR{2831460 (2012j:53053)}

\end{thebibliography}
\end{document}